\documentclass[oneside,reqno]{amsart}
 \usepackage{amsmath,amssymb}
 \usepackage{mathrsfs,amsthm,amstext,amscd,url,amssymb,mathtools,faktor}
\usepackage[margin=3.2cm]{geometry}
\usepackage{hyperref}
\usepackage{wrapfig}
\usepackage{graphics,subfigure}	
\usepackage{tabularx}
\usepackage{booktabs}
\usepackage{epigraph}
\usepackage[inline]{enumitem}
\usepackage{appendix}
\usepackage{tikz}

\theoremstyle{plain}
 \newtheorem{theorem}{Theorem}[section]
 \newtheorem{proposition}{Proposition}[section]
 \newtheorem{lemma}{Lemma}[section]
 
\theoremstyle{definition}
 
 \newtheorem{definition}{Definition}[section]
\theoremstyle{remark}
 \newtheorem{remark}{Remark}[section]
 \numberwithin{equation}{section}
 
\def\1{{\mathchoice {\rm 1\mskip-4mu l} {\rm 1\mskip-4mu l}
 		{\rm 1\mskip-4.5mu l} {\rm 1\mskip-5mu l}}}
\renewcommand{\1}{\boldsymbol 1}

\makeatletter
\newsavebox{\@brx}
\newcommand{\llangle}[1][]{\savebox{\@brx}{\(\m@th{#1\langle}\)}%
	\mathopen{\copy\@brx\kern-0.5\wd\@brx\usebox{\@brx}}}
\newcommand{\rrangle}[1][]{\savebox{\@brx}{\(\m@th{#1\rangle}\)}%
	\mathclose{\copy\@brx\kern-0.5\wd\@brx\usebox{\@brx}}}
\makeatother

\newcommand{\closure}[2][3]{%
	{}\mkern#1mu\overline{\mkern-#1mu#2}}

\newcommand{\BEP}{\text{\normalfont BEP}}

\newcommand{\SIP}{\text{\normalfont SIP}}
\newcommand{\KMP}{\text{\normalfont KMP}}
\newcommand{\BMP}{\text{\normalfont BMP}}

\newcommand{\SEP}{\text{\normalfont SEP}}

\newcommand{\Z}{\mathbb Z}
\newcommand{\R}{\mathbb R}
\newcommand{\N}{\mathbb N}

\newcommand{\E}{\mathbb E}

\renewcommand{\Pr}{\mathbb P}

\newcommand{\scale}{\vartheta}
\newcommand{\tune}{\beta}

\newcommand{\comp}{\text{\normalfont c}}

\newcommand{\bulk}{\text{\normalfont bulk}}
\newcommand{\NegBin}{\text{\normalfont NegBin}}
\newcommand{\dd}{\text{\normalfont d}}
\newcommand{\VN}{\Lambda_N}
\newcommand{\VNh}{\widehat{\Lambda}_N}

\newmuskip\pFqmuskip
\newcommand*\pFq[6][8]{%
	\begingroup 
	\pFqmuskip=#1mu\relax
	\mathcode`\,=\string"8000
	\begingroup\lccode`\~=`\,
	\lowercase{\endgroup\let~}\pFqcomma
	{}_{#2}F_{#3}{\left[\genfrac..{0pt}{}{#4}{#5};#6\right]}%
	\endgroup
}
\newcommand{\pFqcomma}{\mskip\pFqmuskip}

\allowdisplaybreaks

\title[Symmetric inclusion process with slow boundary]{Symmetric inclusion process with slow boundary:\\
	 hydrodynamics and hydrostatics}



\author[Franceschini]{\bfseries Chiara Franceschini} 

\address{Center for Mathematical Analysis, Geometry and Dynamical Systems, IST, Universidade de Lisboa, 1049-001 Lisboa, Portugal
}
\email{chiara.franceschini@tecnico.ulisboa.pt}

\author[Gon\c{c}alves]{Patr\'{i}cia Gon\c{c}alves}
\address{Center for Mathematical Analysis, Geometry and Dynamical Systems, IST, Universidade de Lisboa, 1049-001 Lisboa, Portugal}
\email{pgoncalves@tecnico.ulisboa.pt}

\author[Sau]{Federico Sau}
\address{IST Austria, Am Campus 1, 3400, Klosterneuburg, Austria}
\email{federico.sau@ist.ac.at}

\thanks{\textit{Acknowledgments.} F.S.\ wishes to  thank Joe P.\ Chen for some fruitful discussions at an early stage of this work. C.F. and P.G. thank  FCT/Portugal for support through the project 
	UID/MAT/04459/2013.  This project has received funding from the European Research Council (ERC) under  the European Union's Horizon 2020 research and innovative programme (grant agreement   No.\ 715734). F.S. thanks   CAMGSD, IST, Lisbon, where part of this work has been done, and the European research and innovative programme No.\ 715734 for the kind hospitality. F.S.\ was founded by the European Union's Horizon 2020 research and innovation programme under the Marie-Sk\l{}odowska-Curie grant agreement  No.\ 754411.} 

\begin{document}

\begin{abstract}
We study the hydrodynamic and hydrostatic limits of the one-dimensional open symmetric inclusion process  with slow boundary. Depending on the value of the parameter tuning the interaction rate of the bulk of the system with the boundary, we obtain a linear heat equation with either Dirichlet, Robin or Neumann boundary conditions as hydrodynamic equation.  
In our approach, we combine duality and first-second class particle   techniques to reduce the scaling limit of the inclusion process to the limiting behavior of a single, non-interacting, particle. 
\end{abstract}

\maketitle

\section{Introduction}
Among the interacting particle systems employed to study   non-equilibrium phenomena in mathematical statistical physics, the \emph{inclusion process}  is gaining increasing attention (see, e.g., \cite{carinci_duality_2013-1, vafayi_weakly_2014, floreani_boundary2020, bianchi_metastability_2017, jatuviriyapornchai_structure_2019}). In particular, the \emph{symmetric inclusion process} ($\SIP$), introduced in \cite{giardina_duality_2007} as discrete dual of a Gaussian energy process -- known as Brownian momentum process ($\BMP$) --  and further studied in, e.g., \cite{giardina_duality_2009, carinci_duality_2013-1}, can be considered as the \textquotedblleft attractive\textquotedblright\  counterpart of the symmetric exclusion process ($\SEP$). Indeed,  inclusion particles evolve as independent random walks  subject to an attractive -- rather than repulsive --  interaction with  nearest neighbors and, consequently, with no restriction on the maximal number of particles per site.  Moreover, the $\SIP$ is related via some limiting \textquotedblleft thermalization\textquotedblright\ procedures  to the so-called $\KMP$ model \cite{kipnis_heat_1982}, introduced as a   microscopic model of heat	 transport in non-equilibrium.

The research around these stochastic systems mainly focuses on the microscopic structure of the non-equilibrium steady states as well as on their scaling limits, such as the derivation of Fick's law  for the  non-equilibrium steady state  and the proof of local convergence to a Gibbs state, see, e.g., \cite{kipnis_heat_1982, derrida_exact_1993-1}. In this realm, as for the study of relaxation to the stationary non-equilibrium states,  the first rigorous result on the derivation of the macroscopic  equation governing the evolution of the  density profile  dates back to \cite{eyink_hydrodynamics_1990, eyink_lattice_1991}. There the authors study the hydrodynamic and hydrostatic limits for gradient stochastic lattice gas models in a one-dimensional lattice.     Further developments regarding scaling limits of such systems concern the study of -- both dynamic and static --  large deviations and non-equilibrium fluctuations around the hydrodynamic limit for the open $\SEP$, see, e.g.,  \cite{landim_stationary_2006, farfan2011hydrostatics}. All these models yield     parabolic equations with suitable Dirichlet boundary conditions as hydrodynamic equations.

More recently, stochastic models with more general interactions between the bulk of the system  and the reservoirs have been introduced, see, e.g., \cite{baldasso_exclusion_2017, goncalves_hydrodynamics_2019, franco2015equilibrium, goncalves_non-equilibrium_2019, frometa2020boundary, erignoux_hydrodynamic_2018, erignoux2019hydrodynamicsI, erignoux2019hydrodynamicsII, bernardin2020microscopic}. Depending on the  interaction chosen, more general boundary conditions -- e.g.,  Robin, Neumann  or nonlinear boundary conditions --  and more general nonlinear and fractional diffusions have been derived.  For all these models,  the so-called entropy and relative entropy methods (see, e.g., \cite{kipnis_scaling_1999}) play a prominent role, but they  both require   replacement lemmas in order to  close the equations at the microscopic level. Furthermore, in the context of  exclusion processes in which a matrix formulation for the stationary non-equilibrium state is available and for zero-range processes in which the stationary non-equilibrium measures are of product form, explicit formulas for the stationary  correlations simplify the study of hydrostatic and stationary non-equilibrium fluctuations.

For the open $\SIP$, explicit expressions for the stationary correlations are not, in general, known. Furthermore,  the entropy methods do not directly apply to this context because the partition function of  the local Gibbs measures for $\SIP$ -- products of Negative Binomial distributions --  does not satisfy hypothesis [FEM] in \cite{kipnis_scaling_1999} regarding their radius of convergence. Such an assumption,  which is usually met for a wide class of zero-range and exclusion processes, is crucial in both one- and two-block estimates when applying entropy inequalities.

In view of the inapplicability of these  general and robust methods, we base our study on the duality property of the open $\SIP$. Duality in the context of  interacting particle systems has been given a probabilistic \textquotedblleft graphical\textquotedblright\ interpretation (see, e.g., \cite{liggett_interacting_2005-1}) and has been   thoroughly explored from a Lie algebraic and generating function point of view (see, e.g., \cite{giardina_duality_2009, redig_factorized_2018, carinci2019orthogonal}), enriching both the class of models with the duality property and  the space of duality functions for such models. In words, duality for a particle system may be viewed as the property  of having suitably weighted factorial moments evolving according to a closed system of linear  difference equations. Provided that $n \in \N$ is the degree of such moments, it turns out that the corresponding difference operators which govern  their evolution is given by the infinitesimal generator of $n$ interacting particles which follow the same interaction rules as those of the original system.
In the context of open systems in which particles enter and exit the bulk, the duality property still holds with the dual system having purely absorbing boundary. We remark that this picture is in line with its continuum counterpart, namely with the fact that stochastic representations of solutions of parabolic PDEs with Dirichlet boundary conditions are expressed in terms of diffusion processes which run backward in time and stop when hitting the boundary.

In this paper, we consider the open $\SIP$ on a one-dimensional lattice with  nearest-neighbor interactions, whose boundary rates scale with the size of the system. As done in, e.g., \cite{baldasso_exclusion_2017}, we introduce a parameter $\tune \geq 0$ which tunes the speed of these interactions: the higher the value of $\tune$, the slower the interaction. For this particle system, we derive the hydrodynamic and hydrostatic limits for all values of $\tune \geq 0$, obtaining, in particular, linear heat equations with either Dirichlet, Robin or Neumann boundary conditions depending on whether $\tune \in [0,1)$, $\tune =1$ or $\tune \in (1,\infty)$.

Our strategy to derive the hydrodynamic limit may be summarized as follows. First we center our empirical density fields around their stationary part, which, in our case, is explicitly known.  This centering procedure does not appear in previous literature on scaling limits of  open systems and, by applying it also to the limiting fields, symmetry properties of continuum and discrete Laplacians become available and boundary terms in the limiting equation cancel out.  Then, we exploit the linearity of the evolution equations for the  first and second moments of the occupation variables to close the equations for the associated centered empirical density fields. To this purpose, a \textquotedblleft corrected empirical density field\textquotedblright\ argument (see, e.g., \cite{jara_quenched_2008}) and the centering with respect to the stationary part are crucial in order to close the equations for the fields  and avoid technical replacement lemmas -- based, ultimately, on relative entropy estimates --  as done in the context of the open $\SEP$ in, e.g., \cite{baldasso_exclusion_2017, goncalves_hydrodynamics_2019}. 

For what concerns the hydrostatic limit, in order to verify that the stationary non-equilibrium measures satisfy the assumptions of the hydrodynamic limit, in general, one needs  to control the stationary two-point correlations. This has been done for the slow-boundary $\SEP$ in \cite{baldasso_exclusion_2017} (see also \cite{landim_stationary_2006}) by using the explicit form of such correlations, and in \cite{tsunoda2019hydrostatic} (see also \cite{landim_tsunoda2018}) by proving  replacement lemmas near the boundary. As already mentioned, for $\SIP$,  matrix formulations of the non-equilibrium steady state, explicit formulas for the correlations and replacement lemmas are not available. In order to overcome this, we develop a self-contained method -- which is one of the main contributions of this work --  to derive hydrostatic limits  based solely on duality and a hierarchical representation of the dual particle system. More specifically, we express correlations in terms of a dual system of \emph{two inclusion particles} as  in \cite{floreani_boundary2020}, and combine this with the introduction of a hierarchical \textquotedblleft first-second class particle\textquotedblright\ construction for $\SIP$. This allows us to  reduce the problem of checking the $L^1$-decay of these two-point stationary correlations to the study of a \emph{single one-dimensional  random walk} with absorbing boundary,  considerably simplifying the analysis.

Duality techniques are not new in the context of scaling limits for non-equilibrium systems and have been used thoroughly -- even without explicit mention (see, e.g., \cite{landim_stationary_2006, baldasso_exclusion_2017, erignoux_hydrodynamic_2018}). In this paper, we show that first and second order dualities combined with a purely probabilistic \textquotedblleft lookdown\textquotedblright\ construction of the dual system provide a simple strategy to obtain both hydrodynamic and hydrostatic limit, avoiding both non-homogeneous evolution equations for the two point correlations  and replacement lemmas as, e.g., in \cite{baldasso_exclusion_2017}.

We emphasize that the duality property has to be considered an \textquotedblleft exact\textquotedblright\ feature. Indeed, although some notions of approximate duality proved to be useful in some perturbative contexts (see, e.g., \cite[Chapter 6]{de_masi_mathematical_1991}), in general, duality does not transfer  to perturbations of the particle systems, which, for instance, introduce asymmetries. Nevertheless,  the duality property  is robust with respect to the generalizations of the underlying geometries, as, e.g., with respect to the introduction of disorder or the dimension of the lattice, as well as, of the reservoir interaction  (see, e.g., \cite{floreani_boundary2020}). Moreover, as for the stationary two-point correlations, the aforementioned reduction from two to one dual particles is general and holds for all geometries and reservoir interactions, even when the  stationary particle density profile is not explicitly known.  In this sense, we believe our techniques to  apply to a larger family of discrete and continuum open systems for which analogous duality relations (see \cite{carinci_duality_2013-1}) as well as  hierarchical constructions for the dual processes hold. Among these models, we mention  the  open  symmetric exclusion process $\SEP(\alpha)$, where up to $\alpha \in \N$ particles are allowed to each site (see, e.g., \cite{carinci_duality_2013-1}, whose jump rates differ from those of  the non-gradient generalized exclusion process studied, e.g., in \cite{kipnis_scaling_1999}) and the continuum $\BEP$ and $\KMP$ models.

We further observe that, by following our approach, the regime of fast boundary -- corresponding to $\tune < 0$, still remains open because a control uniform in time on the total mass of particles in the bulk is not at hand for the open  $\SIP$ as in the case, e.g., of  $\SEP$ and zero-range processes. Indeed,  the total mass of the system can be  uniformly  dominated due to hard core constraints  for  the open $\SEP$ and due to the monotonicity -- or, attractiveness --  of the zero-range process (see, e.g., \cite{frometa2020boundary}), property which is not satisfied	 by $\SIP$.  A second challenge consists in the study of  non-equilibrium and stationary fluctuations as well as  of dynamic and static large deviations around the hydrodynamic and hydrostatic limits, respectively, for the open $\SIP$. This is left for a future work.

We conclude this introduction with a short outline on the organization of  the paper. In Section \ref{section:setting} we introduce the particle system, the associated equilibrium and non-equilibrium measures, the dual process and the duality relations. Moreover, we define the functional setting we use to describe our main results, the hydrodynamic and hydrostatic limits stated in Section \ref{section:pdes}. Section \ref{section:proofs} is devoted to the proofs of the two main results. We conclude the paper with two appendices. In Appendix \ref{appendix:function_spaces} we present a complete and unified construction of the function spaces used for which existence and uniqueness of the solution to the hydrodynamic equations we consider follows at once. In Appendix \ref{appendix:RW}, we prove a result for a one-dimensional random walk  required in the proof of the hydrostatic limit.

\section{Setting}\label{section:setting}
In this section, we introduce the   particle system in contact with reservoirs, the duality properties and its stationary measures. Then, we present  the function spaces and the weak formulation of the limiting hydrodynamic equations. 

\subsection{Open symmetric inclusion process}
Let $N \in \N$ play the role of scaling parameter and $\VN \coloneqq \{1,\ldots, N-1 \}$ be the one-dimensional chain on which the particles hop. 
We define by $\mathcal X_N$ the configuration space given by
$
\mathcal X_N\coloneqq \N_0^{\VN} = \{0,1,\ldots\}^{\VN}\ ,
$
where, for any given $\eta \in \mathcal X_N$ and $x \in \VN$, $\eta(x)$ stands for the number of particles of the configuration $\eta$ at site $x$, referred to as occupation variable at $x$. The stochastic dynamics is described by the infinitesimal generator $\mathcal L^N$ whose action on local functions $f : \mathcal X_N\to \R$ is given by
\begin{align}\label{eq:generator}
	\mathcal L^N f\coloneqq \mathcal L^N_\bulk f  + \mathcal L^N_L f + \mathcal L^N_R f\ ,
\end{align}
where
\begin{align*}
	\mathcal L^N_\bulk f(\eta) \coloneqq&\ N^2\sum_{x\in \VN\setminus\{N-1\}} \left\{\begin{array}{r}
		\eta(x)\left(\alpha+\eta(x+1)\right) \left(f(\eta^{x,x+1})-f(\eta) \right)\\[.15cm]
		+\,\eta(x+1)\left(\alpha+\eta(x) \right)\left(f(\eta^{x+1,x})-f(\eta) \right)
	\end{array} \right\}\ ,\\[.2cm]
	\mathcal L^N_L f(\eta) \coloneqq&\ N^{2-\tune} \left\{\begin{array}{r} \alpha_L\scale_L\left(\alpha+\eta(1) \right)\left(f(\eta^{1,+})-f(\eta) \right)\\[.15cm]
		+\, \eta(1)\left(\alpha_L+\alpha_L\scale_L \right)\left(f(\eta^{1,-})-f(\eta) \right)
	\end{array} \right\}
\end{align*}
and
\begin{align*}
	\mathcal L^N_R f(\eta) \coloneqq&\ N^{2-\tune} \left\{\begin{array}{r} \alpha_R\scale_R\left(\alpha+\eta(N-1) \right)\left(f(\eta^{N-1,+})-f(\eta) \right)\\[.15cm]
		+\, \eta(N-1)\left(\alpha_R+\alpha_R\scale_R \right)\left(f(\eta^{N-1,-})-f(\eta) \right)
	\end{array} \right\}\ .	
\end{align*}
In the above expressions, $\eta^{x,y}$ stands for the configuration obtained from $\eta$ by removing a particle from site $x \in \VN$ (if any) and placing it at site $y \in \VN$, i.e., $\eta^{x,y}\coloneqq \eta-\delta_x+\delta_y \in \mathcal X_N$ with $\delta_x$ denoting the configuration consisting of a single particle at site $x \in \VN$. Furthermore,
\begin{align*}
	\begin{split}
		\eta^{1,+} \coloneqq&\ \eta+\delta_1 \\
		\eta^{1,-} \coloneqq&\ \eta-\delta_1
	\end{split}	
	\begin{split}
		\eta^{N-1,+} \coloneqq&\ \eta+\delta_{N-1}\\
		\eta^{N-1,-} \coloneqq&\  \eta-\delta_{N-1}\ .
	\end{split}
\end{align*}
The parameters $\alpha, \alpha_L, \alpha_R, \scale_L$ and $\scale_R$ are all positive and, while $\alpha$ stands for the bulk site attraction parameter, the others describe the interaction of the system with left and right reservoirs through the ending sites of the chain $\VN$. We remark that interpreting  $\alpha_L$ and $\alpha_R$, resp.\ 
\begin{equation}\label{eq:rhoLR}\rho_L\coloneqq\alpha_L \scale_L\qquad \text{and}\qquad \rho_R\coloneqq\alpha_R \scale_R\ ,
\end{equation} as the reservoirs' attraction, resp.\ reservoirs' particle density, parameters, the jump rates due to the reservoir interaction have exactly the same form as the jump rates in the bulk. 	Lastly, the parameter  $\tune \geq  0$ tunes the intensity of the reservoir interaction (see also Figure \ref{Fig1} below). We warn the reader that above and in what follows, for notational convenience, the dependence on $\tune \geq 0$ is never explicitly mentioned. 

In what follows, for all $\mu$ probability measures on $\mathcal X_N$, we let $\Pr^N_{\mu}$ and $\E^N_{\mu}$ denote the probability law and corresponding expectation of the process with generator $\mathcal L^N$ in \eqref{eq:generator} with initial distribution given by $\mu$. If the initial distribution is a Dirac measure, we will adopt the following shortcut: for all $\eta \in \mathcal X_N$, $\Pr^N_\eta\coloneqq \Pr^N_{\delta_\eta}$ and $\E^N_\eta\coloneqq \E^N_{\delta_\eta}$.

There is an immediate comparison of the open inclusion dynamics with the corresponding open exclusion dynamics: with the additional requirement of setting $\alpha \in \N$, the exclusion dynamics in the bulk is recovered by replacing the plus sign in the jump rates with the negative sign, e.g., $\eta(x)\left(\alpha-\eta(x+1) \right)$ in place of $\eta(x)\left( \alpha+\eta(x+1)\right)$; similarly for what concerns the reservoir interaction with the further restriction  $\scale_L, \scale_R \in (0,1)$.

\begin{remark}[\textsc{notation}]
	An alternative parametrization (see, e.g., \cite{carinci_duality_2013-1, goncalves_hydrodynamics_2019} for the $\SEP$)  of the open $\SIP$ employs positive parameters $a_L, a_R, b_L$ and $b_R$ as follows:
	\begin{align*}
		\mathcal L^N_Lf(\eta)
		\coloneqq&\ N^{2-\tune}\left\{\begin{array}{r}b_L \left(\alpha+\eta(1)\right)\, \left(f(\eta^{1,+})-f(\eta)\right)\\[.15cm] 
			+\, a_L\, \eta(1) \left(f(\eta^{1,-})-f(\eta)\right)\end{array} \right\}
	\end{align*}
	and
	\begin{align*}
		\mathcal L^N_Rf(\eta)
		\coloneqq&\ N^{2-\tune}\left\{\begin{array}{r}b_R \left(\alpha+\eta(N-1)\right) \left(f(\eta^{N-1,+})-f(\eta)\right)\\[.2cm]
			+\, a_R\, \eta(N-1) \left(f(\eta^{N-1,+})-f(\eta)\right)\end{array} \right\}\ ,
	\end{align*}
	which corresponds to setting
	\begin{align*}
		\begin{split}
			a_L &= \alpha_L\left(1+\scale_L\right)\\
			a_R &= \alpha_R\left(1+\scale_R\right)
		\end{split}
		\begin{split}
			b_L &= \alpha_L \scale_L\\
			b_R &= \alpha_R \scale_L\ .	
		\end{split}
	\end{align*}
\end{remark}
\begin{figure}
	\includegraphics[width=1.0\textwidth]{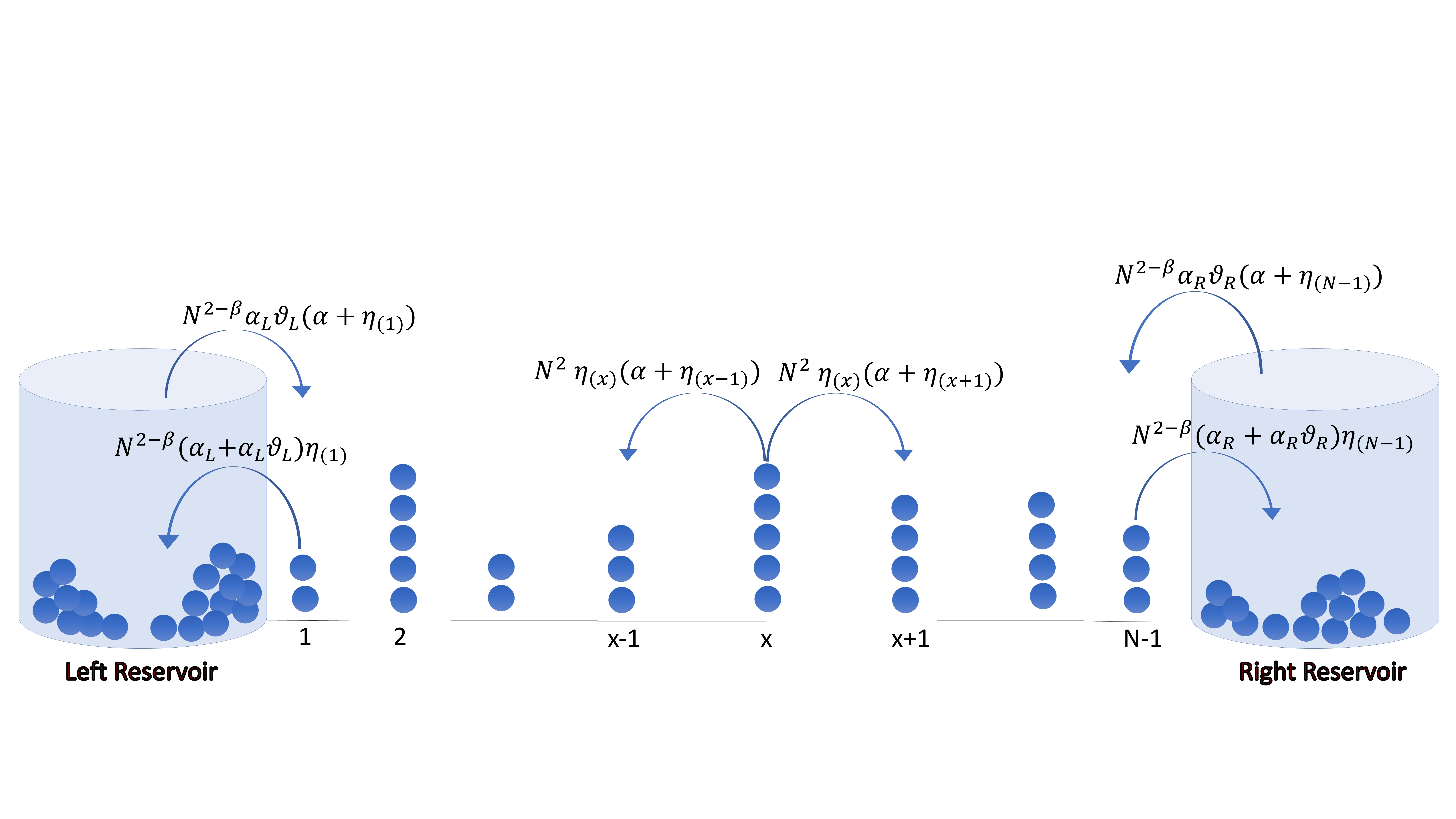}
	\caption[Open $\SIP$]{Schematic description of the dynamics of the one-dimensional open $\SIP$ on $\VN$ with parameters $\alpha, \alpha_L, \alpha_R, \scale_L, \scale_R > 0$ and $\tune \geq 0$.}
	\label{Fig1}
\end{figure}

The bulk dynamics for the open $\SIP$ is conservative and the total number of particles changes only due to particle injection and absorption of the reservoirs.
Unlike   the exclusion process for which each site may be occupied by at most a finite number of particles, for the inclusion process the occupation variables   admit no  prescribed upper bound. Nevertheless, in view of the form of the boundary interaction rates and  classical results on birth-death processes,  the particle system does not explode,  ensuring its existence for any finite initial configuration and  any time. 
\begin{proposition}[\textsc{non-explosiveness}]
	For all $N \in \N$,  $\tune \geq 0$ and  initial configurations $\eta \in \mathcal X_N$, the open $\SIP$ $\{\eta^N_t:\, t\geq 0\}$ with generator $\mathcal L^N$ is \emph{non-explosive}; namely, almost surely, in any  bounded interval of time  the system undergoes finitely many transitions and
	\begin{equation*}
		\left\| \eta^N_t \right\|_{\ell_N^1}\ \coloneqq\ \sum_{x \in \VN} \eta^N_t(x)\ <\ \infty\ ,\quad t \geq 0\ .
	\end{equation*}		
\end{proposition}
\begin{proof}
	The stochastic process
	\begin{equation}\label{eq:norm_eta}	
		\left\{\left\|\eta^N_t\right\|_{\ell^1_N}:\, t \geq 0 \right\}
	\end{equation}
	on $\N_0$ is stochastically dominated by the pure birth process on $\N_0$ started from $\left\|\eta^N_0\right\|_{\ell_N^1}= \left\|\eta\right\|_{\ell_N^1}$ and with birth rates $\left\{r_n:\, n \in \N_0 \right\}$ given by 
	\begin{equation*}
		r_n\coloneqq N^{2-\tune}(\alpha_L\scale_L+\alpha_R\scale_R)\, (\alpha + n)\ ,\quad n \in \N_0\ .
	\end{equation*}
	Since
	$
	\sum_{n \in \N_0} \frac{1}{r_n} =	 \infty
	$,
	such birth process  is non-explosive and, thus, by stochastic domination, also the process in \eqref{eq:norm_eta}.
\end{proof}

\subsubsection{Stationary equilibrium and non-equilibrium measures}\label{section:stationary_measures}
In absence of reservoirs, the $\SIP$ admits a one-parameter family of reversible product measures with marginals given by Negative Binomials with shape parameter $\alpha > 0$ (see, e.g., \cite{giardina_duality_2009, carinci_duality_2013-1}):
\begin{equation}\label{eq:reversible}
	\left\{\mu^N_\scale \coloneqq \otimes_{x \in \VN}\, \nu_\scale:\, \scale > 0\right\}\quad \text{with}\quad \nu_\scale \sim \NegBin(\alpha,\tfrac{\scale}{1+\scale})\ ,
\end{equation}
where our parametrization of $\nu_\scale$ is such that, for all $x \in \VN$,
\begin{align*}
	E_{\mu^N_\scale}\left[\eta(x) \right] =  \scale\alpha\quad \text{and}\quad E_{\mu^N_\scale}\left[\left(\eta(x)-\alpha\scale \right)^2 \right]=  \scale\left(1+\scale\right)\alpha\ .
\end{align*}
Here and in the sequel, for all $\mu$ probability measures,  $E_\mu$ denotes expectation with respect to $\mu$.  In presence of reservoirs, there exists a unique stationary measure $\mu^N_{\scale_L,\scale_R}$ and, depending on the values of $\scale_L$ and $\scale_R$, two different scenarios occur (see, e.g., \cite{carinci_duality_2013-1, floreani_boundary2020} for more details and proofs):
if $\scale_L = \scale_R = \scale > 0$, the system is in equilibrium and the unique stationary -- actually reversible --  measure $\mu^N_{\scale_L,\scale_R}$ is given by $\mu^N_\scale$ in \eqref{eq:reversible}, thus, is product and independent of the parameters $\tune, \alpha_L$ and $\alpha_R$. If $\scale_L\neq \scale_R$, the system is out of equilibrium and the unique stationary measure is not in product form, does depend on $\tune, \alpha_L, \alpha_R$ and it is only partially characterized (see \cite{floreani_boundary2020}). Indeed,  no matrix formulation as for the open exclusion (see, e.g., \cite{derrida_exact_1993-1}) is available for the inclusion process and, hence, two-point (and higher order) correlations are not, in general, explicit. However, to the purpose of deriving the hydrostatic limit  for the open inclusion process, the partial characterization provided in \cite{floreani_boundary2020} plays a crucial role.

\subsection{Duality}\label{section:duality} The duality property will be a key ingredient for all our results. In words, duality for a pair of Markov processes consists in finding an observable -- the so-called \emph{duality function} --  of the joint system whose expectation with respect to the evolution of one marginal equals the expectation with respect to the evolution of the second marginal. In the context of interacting particle systems, duality typically relates the expectation of suitable $n$-joint  moments of the occupation variables of one system with the evolution of $n$ dual interacting particles. Moreover, in presence of reservoirs, duality relates open  systems to dual particle systems with  purely absorbing boundary. Such a correspondence is  related to the well-known Feynman-Kac formulas for parabolic solutions to PDEs with boundary conditions (see, e.g., \cite[Chapter 9]{oksendal_stochastic_1998}), where, in the context of interacting particle systems, this Feynman-Kac formula holds not only for the expected density of particles, but also for suitable higher order  moments. 

\subsubsection{Absorbing process and duality relation}\label{section:dual_process}
Before introducing the duality function, let us describe the absorbing symmetric inclusion process, dual to the process with generator $\mathcal L^N$ defined in \eqref{eq:generator}. For such a dual process, particles evolve on the extended lattice $\VNh\coloneqq \{0,1,\ldots, N\}$ and we let $\widehat {\mathcal X}_N\coloneqq \N_0^{\VNh}$ denote the dual configuration space. The infinitesimal generator of the dual process,  $\widehat {\mathcal L}^N$, is given, for all local functions $f : \widehat{\mathcal X}_N\to \R$, by
\begin{align}\label{eq:generator_dual}
	\widehat{\mathcal L}^Nf\coloneqq \widehat{\mathcal L}^N_\bulk f+\widehat{\mathcal L}^N_Lf + \widehat{\mathcal L}^N_R f\ ,
\end{align}
where the bulk dynamics coincides with that of the open $\SIP$, namely, for all $\xi \in \widehat{\mathcal X}$, 
\begin{align*}
	\widehat{\mathcal L}^N_\bulk f(\xi) \coloneqq&\ N^2\sum_{x\in \VN\setminus\{N-1\}} \left\{\begin{array}{r}
		\xi(x)\left(\alpha+\xi(x+1)\right) \left(f(\xi^{x,x+1})-f(\xi) \right)\\[.15cm]
		+\,\xi(x+1)\left(\alpha+\xi(x) \right)\left(f(\xi^{x+1,x})-f(\xi) \right)
	\end{array} \right\}\ ,
\end{align*}
while the dynamics at left and right ends of $\VNh$ is purely absorbing: for all $\xi \in \widehat {\mathcal X}_N$, 
\begin{align*}
	\widehat{\mathcal L}^N_L f(\xi) \coloneqq&\ N^{2-\tune}  \alpha_L\, \xi(1) \left(f(\xi^{1,0})-f(\xi) \right)\\
	\widehat{\mathcal L}^N_R f(\xi) \coloneqq&\ N^{2-\tune}  \alpha_R\, \xi(N-1)\left(f(\xi^{N-1,N})-f(\xi) \right)\ .	
\end{align*}
We observe that this stochastic dynamics conserves the total number of particles in the system. Moreover, for all $\widehat \mu$ probability measures on $\widehat{\mathcal X}_N$, we let $\widehat \Pr^N_{\widehat \mu}$ and $\widehat \E^N_{\widehat\mu}$ denote the probability law and corresponding expectation of the process with generator $\widehat{\mathcal L}^N$ in \eqref{eq:generator_dual} with initial distribution given by $\widehat \mu$. For notational convenience, for all $\xi \in \widehat{\mathcal X}_N$, $\widehat \Pr^N_\xi\coloneqq \widehat \Pr^N_{\delta_\xi}$ and $\widehat \E^N_\xi\coloneqq \widehat \E^N_{\delta_\xi}$.

\

Let us define the following  function $\mathcal D_N: \widehat{\mathcal X}\times \mathcal X\to \R$ given by
\begin{equation}\label{eq:duality_functions}
	\mathcal D_N(\xi,\eta)\ \coloneqq\ \left(\scale_L \right)^{\xi(0)} \left(\prod_{x\in \VN}d(\xi(x),\eta(x))\right) \left(\scale_R \right)^{\xi(N)}\ ,
\end{equation}
where
\begin{equation*}
	d(k,n)\ \coloneqq\ \tfrac{n!}{(n-k)!} \tfrac{\Gamma(\alpha)}{\Gamma(\alpha+k)}\1_{\{k \leq	 n\}}\ ,\quad k,n \in \N_0\ .
\end{equation*}
We remark that, for all $n \in \N_0$, 
\begin{align*}
	d(0,n) =& 1\ ,\quad
	d(1,n) = \tfrac{n}{\alpha}\quad \text{and}\quad
	d(2,n) = \tfrac{n \left(n-1 \right)}{\alpha \left(\alpha+1 \right)}\ ,
\end{align*}
and, more generally, $d(k,n)$ is a weighted $k$-th falling factorial for the $n$-variable. Moreover, we will need the following property concerning factorial moments of Negative Binomial distributions: for all $k \in \N_0$ and $\scale > 0$,
\begin{align}\label{eq:single_site_duality_integrals}
	E_{\nu_\scale}\left[d(k,\cdot)\right]=\sum_{n \in \N_0} d(k,n)\, \nu_\scale(n)= \scale^k\ .
\end{align}

It was shown in \cite{carinci_duality_2013-1}  that the open and absorbing $\SIP$ are dual with the function $\mathcal D_N$ in \eqref{eq:duality_functions} as duality function, i.e., the following identity -- referred to as \emph{duality relation},
\begin{align}\label{eq:duality_relation}
	\mathcal L^N \mathcal D_N(\xi,\cdot)(\eta) = \widehat {\mathcal L}^N \mathcal D_N(\cdot,\eta)(\xi)
\end{align}
holds for all $N \in \N$, $\tune \geq 0$, $\eta \in \mathcal X_N$ and $\xi \in \widehat{\mathcal X}_N$.  We note that the dual system stochastic dynamics does not depend on the parameters $\scale_L$ and $\scale_R$,  while the duality function $\mathcal D_N$  does. By Kolmogorov equations, the infinitesimal relation \eqref{eq:duality_relation} establishes that, for all $\eta \in \mathcal X_N$, the function
\begin{align*}
	(t,\xi) \mapsto  \E^{N}_\eta\left[\mathcal D_N(\xi,\eta_t) \right] 
\end{align*}
is the solution of the following deterministic linear Cauchy  problem:
\begin{align*}
	\left\{\begin{array}{rclll}
		\frac{d}{dt} f(t,\xi)&=&\widehat{\mathcal L}^N f(t,\xi) &,&\xi \in \widehat{\mathcal X}_N\ ,\ t \geq 0\\[.15cm]
		f(0,\xi) &=& \mathcal D_N(\xi,\eta) &,&\xi \in \widehat{\mathcal X}_N\ .
	\end{array}
	\right.
\end{align*}

\subsubsection{One and two dual particles}
For the sequel, it will be important to express the duality relation in \eqref{eq:duality_relation} and its consequences in terms of \emph{labeled} dual particles. The two cases of interest are those in which the dual system consists of either one or two particles only. 

For what concerns the case of just one particle,  \eqref{eq:duality_relation} rewrites as 
\begin{align}\label{eq:duality_one}
	\mathcal L^N D_N(x,\cdot)(\eta) = A^N D_N(\cdot,\eta)(x)\ ,
\end{align}
where, for all $\eta \in \mathcal X_N$,	
\begin{align*}
	D_N(x,\eta)\coloneqq \mathcal D_N(\delta_x,\eta)= \begin{dcases}
		\tfrac{\eta(x)}{\alpha} &\text{if}\ x \in \VN\\
		\scale_L &\text{if}\ x =0\\
		\scale_R &\text{if}\ x = N\ ,
	\end{dcases}
\end{align*}
and $A^N$ is the generator of a single -- thus, non-interacting --  particle on $\VNh$ with the two endpoints $\{0,N\}$ being absorbing: for all $f: \VNh\to \R$, 
\begin{align}\label{eq:generator_A}
	\nonumber
	A^Nf(x)\coloneqq&\ \1_{\{x \in \VN\}}N^2\sum_{y \in \VN} \1_{\{|y-x|=1\}}\alpha \left(f(y)-f(x)\right)\\ 
	\nonumber
	+&\ \1_{\{x=1\}} N^{2-\tune}   \alpha_L \left(f(0)-f(1)\right)\\[.15cm]
	+&\	  \1_{\{x=N-1\}}N^{2-\tune}  \alpha_R \left(f(N)-f(N-1)\right)\ . 
\end{align}
Let us observe that, restricted to the subspace  of functions $f: \VNh\to \R$ which equal zero at the boundary $\{0,N\}$, $A^N$ is symmetric, i.e., for all $f, g: \VNh\to \R$ such that $f(0)=f(N)=g(0)=g(N)=0$, we have
\begin{align}\label{eq:symmetry_AN}
	\llangle f, A^Ng\rrangle_N = \llangle  A^Nf, g\rrangle_N\ ,
\end{align}
where
\begin{align}\label{eq:inner_product_1}
	\llangle f, g \rrangle_N \coloneqq \frac{1}{N}  \sum_{x \in \VN } f(x)\, g(x)\, \alpha  \ .	
\end{align}

Regarding the case of a dual system consisting of two particles, analogous considerations hold and \eqref{eq:duality_relation} boils down to
\begin{align}\label{eq:duality_two}
	\mathcal L^ND_N(x,y,\cdot)(\eta)= B^ND_N(\cdot\, ,\cdot,\eta)(x,y)\ ,
\end{align}
where, for all $\eta \in \mathcal X_N$,
\begin{align*}
	D_N(x,y,\eta)\coloneqq \mathcal D_N(\delta_x+\delta_y,\eta)= \begin{dcases}\tfrac{\eta(x)\left(\eta(x)-1 \right)}{\alpha\left(\alpha+1 \right)} &\text{if}\ x=y\in \VN\\
		D_N(x,\eta) D_N(y,\eta) &\text{otherwise}\ , 
	\end{dcases}
\end{align*}
and $B^N$ is the generator of two inclusion particles on $\VNh$ with absorbing sites $\{0,N\}$:
\begin{align}\label{eq:generator_two_particles}
	\nonumber
	B^Nf(x,y)
	\coloneqq&\ A^Nf(\cdot,y)(x) + A^Nf(x,\cdot)(y)\\[.10cm]
	+&\, N^2 \1_{\{x,y\neq 0,N\}} \1_{\{|x-y|=1 \}} \left(\left(f(x,x)-f(x,y)\right)+\left(f(y,y)-f(x,y) \right) \right)\ ,
\end{align}
for all functions $f : \VNh \times \VNh \to \R$. For such functions, let us introduce the following inner product
\begin{align}\label{eq:inner_product2}
	\llangle f, g \rrangle_{N\times N}\coloneqq \frac{1}{N^2} \sum_{x \in \VN} \sum_{y \in \VN} f(x,y)\, g(x,y)\, \alpha \left(\alpha+\1_{\{x=y\}} \right)\ .
\end{align}
On the space of functions which are zero on the boundary of $\VNh \times \VNh$, the generator $B^N$ is symmetric with respect to $\llangle \cdot,\cdot\rrangle_{N\times N}$, i.e., 
\begin{align}\label{eq:symmetryBN}
	\llangle f, B^N g\rrangle_{N\times N} = \llangle  B^Nf,  g\rrangle_{N\times N}
\end{align}
for all $f, g: \VNh\times \VNh\to \R$ such that $f(0,\cdot)=f(N,\cdot)=f(\cdot,0)=f(\cdot,N)\equiv0$ and, analogously, for $g$.

\subsection{Test function spaces}\label{section:function_spaces}
In this section, we present, depending on the values of the parameter $\tune \geq 0$, the test function spaces needed to uniquely characterize the weak solution of the limiting hydrodynamic equations. The test function spaces we consider are nuclear Fr\'{e}chet spaces $\mathscr S$ and the  solutions will take values in their dual space of tempered distributions $\mathscr S'$. The construction is  standard and follows the ideas in, e.g., \cite[Chapter 11]{kipnis_scaling_1999} and \cite[Chapter 1]{kallianpur_xiong_1995} of constructing a nested family of Hilbert spaces $\mathcal H_k$, $k \in \Z$, with $\mathcal H_0 = L^2([0,1])$ and for which the canonical embeddings $\mathcal H_{k+k_\ast} \hookrightarrow \mathcal H_k$ are Hilbert-Schmidt for some $k_\ast \in \N$ and for all $k \in \Z$. The main difference in our context compared to the setting in \cite[Chapter 11]{kipnis_scaling_1999} is that, for different values of $\tune \geq 0$, different self-adjoint extensions  of the  Laplacian -- corresponding to different boundary conditions --  must be employed. We present some essential properties of such spaces in Proposition \ref{proposition:characterization_function_spaces} below and leave the details of their construction to Appendix \ref{appendix:function_spaces} below.

We acknowledge that several choices of   test function spaces (and, thus, of weak solutions to the corresponding PDEs, see Section \ref{section:pdes} below) have been employed in the hydrodynamic limit literature. For instance, in \cite{goncalves_hydrodynamics_2019}, more standard Sobolev spaces satisfying an energy estimate and boundary conditions are considered. By taking the aforementioned nuclear space $\mathscr S$ as space of test functions, we make a different choice. This is mainly motivated by the fact that we aim at a unified setting for both hydrodynamics and fluctuation results (to be considered in a future work). In fact, on the one hand, this setting is certainly considered to be the  natural one for the study of fluctuations (see e.g.\ \cite[Chapter 11]{kipnis_scaling_1999}); on the other hand, hydrodynamic results are available in this same framework in a number of  works, see, e.g., \cite{de_masi_mathematical_1991} and references therein. Moreover, our construction of such spaces is different from  the one used in related publications (see e.g.\ \cite{franco_phase_2013,franco2015equilibrium, goncalves_non-equilibrium_2019, bernardin2020microscopic}): there, the authors first define a candidate space of test functions and then verify, knowing some explicit information on suitable orthonormal bases of eigenfunctions, their nuclear structure. In our approach, we first build such spaces from abstract self-adjoint Laplacians  and then  extract  properties of the test functions, without the need of fully characterizing this space. We believe this latter approach to be best suited for proving  scaling limits on more general  geometries.

In what follows,  we distinguish between three different regimes depending on the values of the parameter $\tune \geq 0$ ($\tune < 1$, $\tune = 1$ and $\tune > 1$) corresponding, respectively, to Dirichlet, Robin and Neumann boundary conditions.  

\begin{proposition}\label{proposition:characterization_function_spaces}
	For each of the three regimes, $\tune < 1$, $\tune =1$ and $\tune > 1$, there exists a nuclear Fr\'{e}chet space $\mathscr S=\mathscr S_\tune$ which continuously embeds into $L^2([0,1])$ and which consists of $\mathcal C^\infty([0,1])$ functions, i.e., smooth functions in $(0,1)$ whose derivatives of all orders admit a continuous extension to $[0,1]$. 
	Moreover, depending on the values of $\tune \geq 0$, the test functions in $\mathscr S$ satisfy the following boundary conditions:	
	
	\

	\noindent \emph{Dirichlet} {\normalfont($\tune < 1$)}.  If $G \in\mathscr S$, then 
	\begin{align}\label{eq:boundary_conditions_dirichlet}
		\left(\frac{\dd^+}{\dd u} \right)^{2\ell}\bigg|_{u=0} G = \left(\frac{\dd^-}{\dd u} \right)^{2\ell}\bigg|_{u=1}G = 0
	\end{align}
	holds for all $\ell \in \N_0$.
	
	\
	
	\noindent \emph{Robin} {\normalfont($\tune=1$)}. 
	If $G \in \mathscr S$, then 
	\begin{align}\label{eq:boundary_conditions_robin}
		\nonumber
		\left(\frac{\dd^+}{\dd u} \right)^{2\ell+1}\bigg|_{u=0} G &= \frac{\alpha_L}{\alpha} \left(\frac{\dd^+}{\dd u} \right)^{2\ell}\bigg|_{u=0} G \\  \left(\frac{\dd^-}{\dd u} \right)^{2\ell+1}\bigg|_{u=1}G &= \frac{\alpha_R}{\alpha}\left(\frac{\dd^-}{\dd u} \right)^{2\ell}\bigg|_{u=1}G
	\end{align}	
	holds for all $\ell \in \N_0$.

	\
	
	\noindent \emph{Neumann} {\normalfont($\tune>1$)}. If $G \in \mathscr S$, then
	\begin{align}\label{eq:boundary_conditions_neumann}
		\left(\frac{\dd^+}{\dd u} \right)^{2\ell+1}\bigg|_{u=0} G = \left(\frac{\dd^-}{\dd u} \right)^{2\ell+1}\bigg|_{u=1}G = 0
	\end{align}
	holds for all $\ell \in \N_0$.
\end{proposition}
We defer the proof of the above proposition to Appendix \ref{appendix:function_spaces}.

\section{Hydrodynamic and hydrostatic limits}\label{section:pdes}
In this section, we make precise  the notion of weak solution to the hydrodynamic equations which we use all throughout (Definition \ref{definition:solutions} below); then, we present the statements of both hydrodynamic and hydrostatic limits for the open $\SIP$ (Theorems \ref{theorem:main_hydrodynamics} and \ref{theorem:main_hydrostatic}, resp.,  below). 

Let us recall that, for all $\tune \geq 0$, $\mathscr S'=\mathscr S'_\tune$ denotes the strong topological dual of $\mathscr S=\mathscr S_\tune$, the space of test functions introduced  in Section \ref{section:function_spaces} above. Since we state our hydrodynamic and hydrostatic  limits in terms of convergence in the space of distributions, we characterize these limits as the unique solutions in $\mathscr S'$ of the following formal partial differential equations:
\begin{equation}\label{eq:cauchy_theta}
	\left\{\begin{array}{rclll}
		\partial_t \scale(t,u) &=& \alpha\, \partial^2_u \scale(t,u) && t \in [0,\infty)\, ,\  u \in (0,1)\\[.2cm]
		\gamma_L\,\partial_u^+\scale(t,0)&=& \lambda_L\,\scale(t,0) + \sigma_L &&  t\in (0,\infty)\\[.2cm]
		\gamma_R\,\partial_u^-\scale(t,1) &=& \lambda_R\, \scale(t,1) +\sigma_R &&  t\in (0,\infty)\\[.2cm]
		\scale(0,u) &=& \scale_0(u) && u\in [0,1]\ ,		
	\end{array}\right.
\end{equation}	
where $\gamma_L, \gamma_R,  \lambda_L, \lambda_R, \sigma_L, \sigma_R \in \R$ are determined according to the value of $\tune\geq 0$ and the system parameters. In particular, if $\tune < 1$, then we will recover Dirichlet boundary conditions with
\begin{equation}
	\gamma_L = \gamma_R = 0\ ,\qquad \lambda_L = \lambda_R = 1\ ,\qquad \sigma_L = \scale_L\ ,\qquad  \sigma_R =\scale_R\ ;	
\end{equation}
if $\tune =1$, we will recover Robin boundary conditions with
\begin{equation}
	\gamma_L = \gamma_R=\alpha\ ,\qquad \lambda_L = \alpha_L\ ,\qquad \lambda_R = \alpha_R\ ,\qquad \sigma_L = -\alpha_L\, \scale_L\ ,\qquad \sigma_R = -\alpha_R\, \scale_R\ ;
\end{equation} if $\tune > 1$, we will recover Neumann boundary conditions with
\begin{equation}
	\gamma_L = \gamma_R = 1\ ,\qquad \lambda_L = \lambda_R = \sigma_L = \sigma_R = 0\ .
\end{equation}
If we let $h$ denote a stationary solution -- not necessarily unique -- of the  boundary Cauchy problem \eqref{eq:cauchy_theta}, then $\scale$ in \eqref{eq:cauchy_theta} above decomposes as $\scale= h+g$, where $g$ formally satisfies
\begin{equation}
	\left\{\begin{array}{rclll}
		\partial_t g(t,u) &=& \alpha\, \partial^2_u g(t,u) && t \in [0,\infty)\, ,\  u \in (0,1)\\[.2cm]
		\gamma_L\, \partial_u^+g(t,0) &=& \lambda_L\, g(t,0) &&  t\in (0,\infty)\\[.2cm]
		\gamma_R\, \partial_u^- g(t,1) &=& \lambda_R\, g(t,1) &&  t\in (0,\infty)\\[.2cm]
		g(0,u) &=& \scale_0(u)-h(u) && u\in [0,1]		\ .
	\end{array}\right.
\end{equation}

Before presenting the precise definition of solutions in $\mathscr S'$, we need to introduce some notation.
For all $G \in \mathscr S$ and $g \in \mathscr S'$, we define $\langle G, g\rangle\coloneqq g(G)$. We note that $L^2([0,1])\subsetneq \mathscr S'$ and, if, e.g.,  $g \in L^2([0,1])$, then  $\langle G, g\rangle$ is the usual inner product in $L^2([0,1])$. Moreover,  we let $\mathcal C([0,\infty),\mathscr S')$ and $\mathcal D([0,\infty),\mathscr S')$ denote the spaces of $\mathscr S'$-valued continuous and \emph{c\`{a}dl\`{a}g}, respectively, functions on $[0,\infty)$ (see, e.g., \cite[\S2.4]{kallianpur_xiong_1995}, as well as, Appendix \ref{appendix:function_spaces} below). 	Finally, for all $\tune \geq 0$, $\mathcal A: \mathscr S\to \mathscr S$ denotes the bounded linear operator introduced in Appendix \ref{appendix:function_spaces} below,
 which acts on smooth functions $G \in \mathcal C^\infty([0,1])$ simply as the rescaled Laplacian
\begin{align*}
	\mathcal A G = \alpha\, \partial^2_u G\ .	
\end{align*}

\begin{definition}[\textsc{Solutions in $\mathscr S'$}]\label{definition:solutions} Let $\tune \geq 0$. 
	Given  	 $\scale_0  \in  \mathscr S'$, 	we say that $\{\scale(t) : t \in [0,\infty) \} \subset \mathscr S'$ is a  solution of the   Dirichlet, Robin or Neumann  problem -- depending on whether $\tune < 1$, $\tune = 1$ or $\tune > 1$, respectively --	with initial condition $\scale_0$ 	if there exists $\left\{g(t): t \in [0,\infty)\right\} \subset \mathscr S'$ for which,   for all $G \in \mathscr S$	 and for all times $t \geq 0$, the following two identities hold:
	\begin{equation}\label{eq:decomposition}
		\langle G, \scale(t) \rangle = \langle G, h_{\scale_L,\scale_R}\dd u\rangle + \langle G, g(t) \rangle\ .
	\end{equation}
	and
	\begin{equation}
		\langle G,g(t)\rangle = \langle G, (\scale_0-h_{\scale_L,\scale_R}\dd u) \rangle + \int_0^t \langle \mathcal A G, g(s)\rangle\, \dd s\ .
	\end{equation}	
	In the above expressions,  $h_{\scale_L,\scale_R}\,\dd u \in \mathscr S'$ is the distribution that is absolutely continuous with respect to Lebesgue and whose density 	  $h_{\scale_L,\scale_R} \in \mathcal C^\infty([0,1])$ is a  stationary solution to \eqref{eq:cauchy_theta}, i.e.,  given by
	\begin{align}\label{eq:harmonic_dirichlet}
		h_{\scale_L,\scale_R}(u) =  \scale_L +  (\scale_R-\scale_L)\, u\ ,\quad u \in [0,1]\ , 
	\end{align}
	if $\tune < 1$,  by
	\begin{align}\label{eq:harmonic_robin}
		h_{\scale_L,\scale_R}(u)= 
		\scale_L +\left(\scale_R-\scale_L \right) \left(\frac{\tfrac{\alpha}{\alpha_L}+u}{\tfrac{\alpha}{\alpha_L}+1+\tfrac{\alpha}{\alpha_R}} \right)\ ,\quad u \in [0,1]\ ,	
	\end{align}
	if $\tune = 1$, and by
	\begin{equation}\label{eq:harmonic_neumann}	
		h_{\scale_L,\scale_R}(u) = \scale_L +\left( \scale_R-\scale_L\right) \frac{\alpha_R}{\alpha_L+\alpha_R}= \frac{\rho_L+\rho_R}{\alpha_L+\alpha_R}\ ,\quad u \in [0,1]\ ,
	\end{equation} 
	if $\tune > 1$.
\end{definition}

As a consequence of the construction of the test function spaces in Section \ref{section:function_spaces} (see also Appendix \ref{appendix:function_spaces}) and the theory of generalized Ornstein-Uhlenbeck processes (see, e.g., \cite{holley_generalized_1978}) applied to this deterministic setting, the following existence and uniqueness result holds. 
\begin{proposition}[\textsc{well-posedness of hydrodynamic equations in $\mathscr S'$}]\label{proposition:well_posedness}
	For all $\tune \geq 0$, the  solution in $\mathscr S'$ with initial condition $\scale_0 \in \mathscr S'$  as defined in Definition \ref{definition:solutions} exists and is unique in $\mathcal C([0,\infty),\mathscr S')$.		
\end{proposition}
\begin{proof}
	By the construction  of the nuclear Fr\'{e}chet space $\mathscr S$ and the consequent properties \ref{it:p1}--\ref{it:p4} in Appendix \ref{appendix:function_spaces}, Theorem 1.23 in \cite{holley_generalized_1978} without noise, i.e.,  taking $B \equiv 0$, applies.
\end{proof}

All throughout, since we state our results in terms of  solutions in $\mathscr S'$ and the investigation of their regularity is not the prominent goal of our work, we refer the interested reader to, e.g., \cite{Evans} for further details, for instance, on the assumptions on the initial condition which guarantee  such solutions to be actually strong ones for \eqref{eq:cauchy_theta}.
Further notice that  the functional framework that we employ allow us to prove our limit theorems (see Theorems \ref{theorem:main_hydrodynamics} and \ref{theorem:main_hydrostatic} below) for  initial profiles which are generalized functions in $\mathscr S'$.

\subsection{Main results} In this section,  we present, for all $\tune \geq 0$,	 the two results concerning the weak law of large numbers -- the hydrodynamic and hydrostatic limits --  for the empirical density fields $\left\{\mathscr X^N_\cdot: N \in \N \right\} \subset \mathcal D([0,\infty),\mathscr S')$, given  by
\begin{equation}\label{eq:density_fieldsN}
	\mathscr X^N_t \coloneqq \frac{1}{N}\sum_{x \in \VN} \delta_{\frac{x}{N}}\, \eta^N_t(x)=  \frac{1}{N} \sum_{x \in \VN} \delta_{\frac{x}{N}}\, D_N(x,\eta^N_t)\, \alpha\ ,\quad t \geq 0\ ,
\end{equation}
where, for all $N \in \N$,  $\left\{\eta^N_t: t \geq 0\right\}$ is the open $\SIP$ with some  prescribed initial distribution	 $\mu^N$ (we refer to the statements of the two main theorems below for further details). 

Before the statement of the hydrodynamics result, we need a further definition.
\begin{definition}[\textsc{particle distributions associated with a profile}]\label{definition:associated}
	For all $\tune \geq 0$, let $\{\mu^N: N \in \N\}$ be a sequence of Borel probability measures on $\left\{\mathcal X_N: N \in \N\right\}$ and let $\scale_0\, \alpha \in \mathscr S'$. We say that the family $\{\mu^N: N \in \N \}$ is \emph{associated with the  profile} $\scale_0\, \alpha \in \mathscr S'$ if, for all $G \in \mathscr S$ and for all $\delta > 0$, 
	\begin{align*}
		\mu^N\left(\left\{\eta \in \mathcal X_N: \left|\frac{1}{N}\sum_{x\in \VN}G(\tfrac{x}{N})\, D_N(x,\eta)\, \alpha - \langle G, \scale_0\,\alpha	\rangle\right|\ >\ \delta \right\}\right) \underset{N\to \infty}\longrightarrow 0\ .
	\end{align*} 
\end{definition}


\begin{theorem}[\textsc{hydrodynamic limit}]\label{theorem:main_hydrodynamics}
	For all $\tune \geq 0$, let $\left\{\mu^N: N \in \N \right\}$ be a family of Borel probability measures on $\left\{\mathcal X_N: N \in \N \right\}$  and let $\scale_0\, \alpha \in \mathscr S'$. We assume that:
	\begin{enumerate}[label={\normalfont (\alph*)}, ref={\normalfont (\alph*)}]
		\item \label{it:assumption_associated}The family $\{\mu^N: N \in \N \}$ is associated with the profile $\scale_0\, \alpha \in \mathscr S'$ (see Definition \ref{definition:associated}).
		\item \label{it:assumption_bounds}	There exists a constant $\kappa > 0$ such that, for all $N \in \N$ and $x, y \in \VNh$, the following upper bounds hold:
		\begin{align}\label{eq:upper_bound_first_second_moments}
			E_{\mu^N}\left[D_N(x,\eta) \right] \leq \kappa\qquad \text{and}\qquad
			E_{\mu^N}\left[D_N(x,y,\eta)\right] \leq \kappa^2\ .	
		\end{align}
	\end{enumerate}
	Let us consider the empirical density fields $\left\{\mathscr X^N_\cdot: N \in \N \right\} \subset \mathcal D([0,\infty),\mathscr S')$  defined as in \eqref{eq:density_fieldsN} in terms of the open  symmetric inclusion processes   initialized according to $\left\{\mu^N: N \in \N \right\}$, i.e., 
	\begin{equation*}
		\eta^N_0 \sim \mu^N\ ,\quad N \in \N\ .
	\end{equation*}
	Then, the following weak convergence in $\mathcal D([0,\infty),\mathscr S')$ (see also \eqref{eq:weak_convergence2} below)
	\begin{align}\label{eq:weak_convergence_hydrodynamic_limit}
		\left\{\mathscr X^N_t: t \geq 0 \right\} \underset{N\to \infty}\Longrightarrow \left\{\scale(t)\,\alpha: t \geq 0 \right\}
	\end{align}
	holds, where
	$\left\{\scale(t)\,\alpha: t \geq 0 \right\} \in \mathcal C([0,\infty),\mathscr S')
	$ is the unique  solution in $\mathscr S'$  of 
	\begin{itemize}
		\item the Dirichlet problem  if $\tune <1$,
		\item the Robin problem  if $\tune =1$,
		\item  the Neumann problem  if $\tune > 1$,
	\end{itemize}  starting from $\scale_0\, \alpha \in \mathscr S'$.  
\end{theorem}
We observe that, if $\scale_0\, \alpha \in \mathscr S'$ is absolutely continuous with respect to Lebesgue with non-negative continuous density $\theta_0\, \alpha \in \mathcal C([0,1])$, then the local Gibbs measures $\left\{\mu^N: N \in \N \right\}$ (see Section \ref{section:stationary_measures}) given by 
\begin{equation*}
	\mu^N= \otimes_{x \in \VN} \nu_{\theta_0(\frac{x}{N})}\ ,\quad N \in \N\ ,	
\end{equation*}
satisfy both assumptions  \ref{it:assumption_associated} with the profile $\scale_0\, \alpha$ and \ref{it:assumption_bounds}  with $\kappa\coloneqq \max\left\{\scale_L,\scale_R, \sup_{u \in [0,1]} \theta_0(u) \right\}$.  

Furthermore, because the transitions of the underlying open particle system consist in only one-particle moves, for all $\tune \geq 0$, $G \in \mathscr S$ and $\delta > 0$,  we have
\begin{align}\label{eq:vanishing_jumps}
	\Pr^N_{\mu^N}\left(\sup_{t \geq 0} \left|\langle G, \mathscr X^N_t\rangle - \langle G, \mathscr X^N_{t^-}\rangle \right|>\delta \right)\underset{N\to \infty}\longrightarrow 0\ ,
\end{align}
from which it follows (\cite[Theorem 13.4]{billingsley_convergence_1999}) that any  limiting point $\mathscr X_\cdot$ of the sequence $\left\{\mathscr X^N_\cdot: N \in \N \right\}$ belongs to $\mathcal C([0,\infty),\mathscr S')$. The result in Proposition \ref{proposition:well_posedness} and assumption \ref{it:assumption_associated} will, then, univocally characterize the deterministic limiting process.	 

The weak convergence in \eqref{eq:weak_convergence_hydrodynamic_limit}, which boils down to show tightness and convergence of the finite dimensional distributions of the sequence $\left\{\mathscr X^N_\cdot: N \in \N  \right\}$ (see, e.g., \cite[Proposition 5.2]{mitoma_tightness_1983}), because of the considerations in Section \ref{section:remark_topologies} below and because the limiting process is deterministic and continuous, may be equivalently restated as follows (cf.\ \eqref{eq:convergence_probability} below, as well as,  e.g., \cite{kallianpur_xiong_1995} for further details): for all $T \geq 0$, $G \in \mathscr S$ and $\delta >0$, 
\begin{align}\label{eq:weak_convergence2}
	\Pr^N_{\mu^N}\left(\sup_{t \in [0,T]}\left| \langle G,\mathscr X^N_t\rangle - \langle G, \scale(t)\, \alpha\rangle\right|>\delta \right)\underset{N\to \infty}\longrightarrow 0\ .
\end{align}

The general strategy we follow to prove Theorem \ref{theorem:main_hydrodynamics} is to, first,   provide a decomposition for the empirical density fields analogous to that in \eqref{eq:decomposition} in which we center the  fields  with respect to  the stationary part, i.e., write, for all $N \in \N$ and $t \geq 0$,
\begin{align}\label{eq:decomposition_N}
	\mathscr X^N_t = \mathscr H^N_{\scale_L,\scale_R}\, \alpha + \mathscr  Z^N_t
\end{align}
for $\mathscr H^N_{\scale_L \scale_R}  \in \mathscr S'$ deterministic and $\left\{\mathscr Z^N_t: t \geq 0 \right\} \in \mathcal D([0,\infty),\mathscr S')$ random; then, show that  
\begin{align}\label{eq:convergence_stationary_part}
	\langle G, \mathscr H^N_{\scale_L,\scale_R}\,\alpha\rangle \underset{N\to \infty}\longrightarrow \langle G, h_{\scale_L,\scale_R}\, \alpha \,  \dd u\rangle
\end{align}
and
\begin{align}\label{eq:convergence_dynamic_part}
	\Pr^N_{\mu^N}\left(\sup_{t \in [0,T]}\left|\langle G, \mathscr Z^N_t\rangle - \langle G, \left(\scale(t)-h_{\scale_L,\scale_R}\dd u\right)\alpha\rangle \right| > \delta \right)\underset{N\to \infty}\longrightarrow 0
\end{align}
hold for all  $T \geq 0$,  $G \in \mathscr S$ and $\delta > 0$.

\

In the following theorem, we present our second main result and recall that, for all $\tune \geq 0$ and $N \in \N$, 
$\mu^N_{\scale_L,\scale_R}$ denotes the unique stationary probability measure  for the open $\SIP$ $\left\{\eta^N_t: t\geq 0 \right\}$ (see Section \ref{section:stationary_measures}).

\begin{theorem}[\textsc{hydrostatic limit}]\label{theorem:main_hydrostatic}
	For all $\tune \geq 0$, the empirical density fields  -- given in \eqref{eq:density_fieldsN} and defined in terms of the open $\SIP$ initialized according to $\left\{\mu^N_{\scale_L,\scale_R}: N \in \N\right\}$ --   weakly converge in $\mathcal D([0,\infty),\mathscr S')$ to
	$
	\left\{\scale(t)\, \alpha: t \geq 0 \right\} \in \mathcal C([0,\infty),\mathscr S')$,	
	where
	$
	\scale(t) \equiv h_{\scale_L,\scale_R}\, \dd u \in \mathscr S'
	$ and $h_{\scale_L,\scale_R} \in \mathcal C^\infty([0,1])$ is the unique stationary solution
	of
	\begin{itemize}
		\item the Dirichlet problem  as given in \eqref{eq:harmonic_dirichlet} if $\tune <1$,
		\item the Robin problem  as given in \eqref{eq:harmonic_robin} if $\tune =1$,
		\item  the Neumann problem   as given in \eqref{eq:harmonic_neumann} if $\tune > 1$.
	\end{itemize}
\end{theorem}

\section{Proofs}\label{section:proofs}
Let us now  prove Theorems \ref{theorem:main_hydrodynamics} and \ref{theorem:main_hydrostatic}.
Before digging into the proofs, though, we start with some general considerations. 

\

Because of duality \eqref{eq:duality_relation} and  because the dual process defined in Section \ref{section:dual_process} conserves the total number of particles, in view of assumption \ref{it:assumption_bounds} in Theorem \ref{theorem:main_hydrodynamics}, we have
\begin{align}\label{eq:bound_second_order_polynomials}
	\sup_{N \in \N}\sup_{t \geq 0}\sup_{x \in \VNh} \E^N_{\mu^N}\left[D_N(x,\eta^N_t) \right]\leq \kappa\quad \text{and}\quad
	\sup_{N \in \N}\sup_{t \geq 0}\sup_{x, y \in \VNh} \E^N_{\mu^N}\left[D_N(x, y,\eta^N_t) \right]\leq \kappa^2\ .
\end{align}

\

We recall the definition of the empirical density fields in \eqref{eq:density_fieldsN}. In view of Dynkin's formula, for all $\tune \geq 0$, $N \in \N$, $t \geq 0$ and $G \in \mathscr S$, we have
\begin{align}\label{eq:dynkin_formula}
	\langle G,\mathscr X^N_t\rangle = \langle G, \mathscr X^N_0\rangle + \int_0^t \mathcal L^N \langle G,\mathscr X^N_s\rangle\, \dd s  + \langle G,\mathscr M^N_t\rangle\ ,
\end{align}
where $\left\{\mathscr M^N_t: t \geq 0 \right\} \subset \mathcal D([0,\infty),\mathscr S')$ is a martingale (with respect to its natural filtration) with predictable quadratic variation given, for all $t \geq 0$, by
\begin{align*}
	\int_0^t \left( \mathcal L^N \left(\langle G, \mathscr X^N_s \rangle\right)^2 - 2 \langle G,\mathscr X^N_s\rangle \mathcal L^N \langle G, \mathscr X^N_s\rangle\right) \dd s\ .	
\end{align*}

\

We recall the definition of the inner product $\llangle \cdot,\cdot \rrangle_N$ from \eqref{eq:inner_product_1} and of the generator $A^N$ in \eqref{eq:generator_A}. In view of the duality relations \eqref{eq:duality_relation} and, in particular, \eqref{eq:duality_one} and \eqref{eq:duality_two}, we have
\begin{align}\label{eq:drift}
	\mathcal L^N \langle G, \mathscr X^N_s\rangle =&\ \frac{1}{N}\sum_{x \in \VN} G(\tfrac{x}{N})\, A^N D_N(\cdot,\eta^N_s)(x)\, \alpha = \llangle G(\tfrac{\cdot}{N}),A^ND_N(\cdot,\eta^N_s)\rrangle_N
\end{align} 
and
\begin{align}\label{eq:predictable_QV}\nonumber
	&\int_0^t \left(\mathcal L^N \left(\langle G, \mathscr X^N_s\rangle \right)^2 - 2 \langle G, \mathscr X^N_s \rangle \mathcal L^N \langle G, \mathscr X^N_s\rangle\right) \dd s\\
	=&\ \frac{1}{N^2} \int_0^t \left\{\begin{array}{c} \sum_{x \in \VN\setminus \{N-1\}} N^2 \alpha \left(G(\tfrac{x+1}{N})-G(\tfrac{x}{N}) \right)^2 \mathcal V^N_{\{x,x+1\}}(\eta^N_s)\, \alpha\\[.2cm]
		+\, N^{2-\tune} \alpha_L \left(G(\tfrac{1}{N}) \right)^2 \mathcal V^N_{\{0,1\}}(\eta^N_s)\, \alpha\\[.2cm]
		+\, N^{2-\tune} \alpha_R \left(G(\tfrac{N-1}{N}) \right)^2 \mathcal V^N_{\{N-1,N\}}(\eta^N_s)\, \alpha
	\end{array} \right\}\dd s\ ,
\end{align}
where, for all $\eta \in \mathcal X_N$ and $x, y \in \VNh$ with $x\neq y$, 
\begin{align*}
	\mathcal V^N_{\{x,y\}}(\eta)\coloneqq&\ D_N(x,\eta) + D_N(y,\eta) + 2 D_N(x,y,\eta)\ .
\end{align*}
Going back to \eqref{eq:drift}, we note that, for all $f: \VNh\to \R$, two applications of integration by parts yield
\begin{align}\label{eq:computations_laplacian}
	\llangle G(\tfrac{\cdot}{N}), A^Nf\rrangle_N =&\	\frac{1}{N} \sum_{x \in \VN} \alpha\, \Delta_N G(\tfrac{x}{N})\, f(x)\, \alpha\\
	\nonumber
	+&\   \alpha\, \nabla_N^+ G(0)\, f(1)\,\alpha + \alpha\, \nabla_N^- G(1)\, 	f(N-1)\,\alpha\\[.2cm] 
	+&\  N^{1-\tune}\alpha_L\, G(\tfrac{1}{N})\left(f(0)-f(1) \right)\alpha +  N^{1-\tune}\alpha_R\, G(\tfrac{N-1}{N})\left(f(N)-f(N-1) \right) \alpha\ ,
	\nonumber
\end{align}
where $\Delta_N$ denotes the discrete Laplacian with mesh size $\frac{1}{N}$, namely
\begin{align*}
	\Delta_N G(\tfrac{x}{N})\coloneqq N^2 \left(G(\tfrac{x+1}{N})+G(\tfrac{x-1}{N})-2 G(\tfrac{x}{N}) \right)\ ,\quad x \in \VN\ ,
\end{align*}
and $\nabla^\pm_N$ the corresponding discrete gradients:
\begin{align*}
	\nabla^+_N G(\tfrac{x}{N})\coloneqq&\ N\left(G(\tfrac{x+1}{N})-G(\tfrac{x}{N})\right)\ ,\quad x \in \VNh \setminus\{N\}\\
	\nabla^-_N G(\tfrac{x}{N})\coloneqq&\ N\left(G(\tfrac{x-1}{N})-G(\tfrac{x}{N})\right)\ ,\quad x \in \VNh \setminus\{0\}\ .
\end{align*}
We recall from \eqref{eq:symmetry_AN} that if, additionally, $G(0)=G(1)=0$ and $f(0)=f(N)=0$, then
\begin{align}\label{eq:symmetry_one}
	\llangle G(\tfrac{\cdot}{N}), A^Nf\rrangle_N = \llangle \mathcal A^N G(\tfrac{\cdot}{N}), f\rrangle_N\ , 
\end{align}
where
\begin{equation}\label{eq:mathcalA^N}
	\mathcal A^N G(\tfrac{x}{N})\coloneqq A^N G(\tfrac{\cdot}{N})(x)\ ,\quad x \in \VNh\ .
\end{equation}

Having in mind the decomposition in \eqref{eq:decomposition_N} of the empirical density fields, we introduce, for all $\tune \geq 0$ and $N \in \N$, the following function
\begin{align}\label{eq:definitionhN}
	h^N_{\scale_L,\scale_R}(x)\coloneqq E_{\mu^N_{\scale_L,\scale_R}}\left[D_N(x,\eta) \right] = \lim_{t\to \infty}\E_{\mu^N}\left[D_N(x,\eta^N_t) \right]\ ,\quad x \in \VNh\ ,
\end{align}
for any probability measure $\mu^N$ on $\mathcal X_N$, which, by stationarity of $\mu^N_{\scale_L,\scale_R}$ and duality \eqref{eq:duality_one}, solves the following boundary value problem:
\begin{align}\label{eq:boundary_value_discrete_one}
	\left\{\begin{array}{rcll}
		A^N f(x) &=& 0\ , &x \in \VNh\\[.2cm]
		f(0) &=& \scale_L\\[.2cm]
		f(N) &=& \scale_R\ .
	\end{array}\right.
\end{align}
Notice that, as we will show in Lemma \ref{lemma:convergence_static_part} below, the functions $h^N_{\scale_L,\scale_R}$ are  to be considered as  discrete approximations of the stationary solutions of the hydrodynamic equations.
Moreover, by defining
\begin{equation}\label{eq:bigHN}\mathscr 	H^N_{\scale_L,\scale_R} \alpha\coloneqq \frac{1}{N}\sum_{x \in  \VN} \delta_{\frac{x}{N}}\, h^N_{\scale_L,\scale_R}(x)\,\alpha \ ,
\end{equation}   \eqref{eq:decomposition_N} writes, for all $\tune \geq 0$, $N \in \N$, $t \geq 0$ and $G \in \mathscr S$,   as 
\begin{align}\label{eq:decomposition_N2}
	\nonumber
	\langle G, \mathscr X^N_t\rangle =&\ \langle G, \mathscr H^N_{\scale_L,\scale_R}\, \alpha\rangle  + \langle G, \mathscr Z^N_t\rangle\\
	\nonumber
	=&\ \llangle G(\tfrac{\cdot}{N}), h^N_{\scale_L,\scale_R}\rrangle_N + \frac{1}{N}\sum_{x \in \VN} G(\tfrac{x}{N})\left(\tfrac{\eta^N_t(x)}{\alpha}-h^N_{\scale_L,\scale_R}(x) \right)\alpha\\
	=&\  \llangle G(\tfrac{\cdot}{N}), h^N_{\scale_L,\scale_R}\rrangle_N + \frac{1}{N}\sum_{x \in \VN} G(\tfrac{x}{N})\left(D_N(x,\eta^N_t)-E_{\mu^N_{\scale_L,\scale_R}}\left[D_N(x,\eta) \right] \right)\alpha\  .
\end{align}
In our one-dimensional context, the explicit form of the function $h^N_{\scale_L,\scale_R}$ is well-known and given by
\begin{align}\label{eq:harmonic_discrete}
	h^N_{\scale_L,\scale_R}(x)= \scale_L + p_N(x) \left(\scale_R-\scale_L \right)\ ,\quad x \in \VNh\ ,
\end{align}
with
\begin{align*}
	p_N(x)\ \coloneqq\ \frac{\1_{\{x>0\}}}{Z^N} \frac{1}{N}\left(\frac{1}{N^{-\tune}\alpha_L \alpha}+ \frac{x-1}{\alpha^2} 	+\1_{\{x=N\}}\frac{1}{N^{-\tune}\alpha_R \alpha}\right)\ ,
\end{align*}
where
\begin{align*}
	Z^N\ \coloneqq\ \frac{1}{N}\left(\frac{1}{N^{-\tune}\alpha_L \alpha}+ \frac{N-1}{\alpha^2}+\frac{1}{N^{-\tune}\alpha_R \alpha}\right)\ .
\end{align*}
In particular, $h^N_{\scale_L,\scale_R}(0)=\scale_L$ and $h^N_{\scale_L,\scale_R}(N)=\scale_R$.
Hence, in order to prove both Theorems \ref{theorem:main_hydrodynamics} and \ref{theorem:main_hydrostatic}, as a first step, we prove the convergence in \eqref{eq:convergence_stationary_part} in the following lemma. 
\begin{lemma}[\textsc{convergence of the stationary part}]\label{lemma:convergence_static_part}
	Let us recall, for all $\tune \geq 0$, the definitions of stationary solutions $h_{\scale_L,\scale_R}$ in \eqref{eq:harmonic_dirichlet}, \eqref{eq:harmonic_robin} and \eqref{eq:harmonic_neumann} as well as the definition of $\mathscr H^N_{\scale_L,\scale_R}$ in \eqref{eq:bigHN}. Then, for all $\tune \geq 0$ and $G \in \mathscr S$, we have 
	\begin{align*}
		\langle G, \mathscr H^N_{\scale_L,\scale_R}\, \alpha\rangle\underset{N\to \infty}\longrightarrow  \langle G, h_{\scale_L,\scale_R}\, \alpha\, \dd u  \rangle= \int_{[0,1]} G(u)\, h_{\scale_L,\scale_R}(u)\, \alpha\,	 \dd u\ .
	\end{align*}
\end{lemma}
\begin{proof} We note that, for all $\tune \geq 0$ and by definition of  $h^N_{\scale_L,\scale_R}$ in \eqref{eq:harmonic_discrete}, we have
	\begin{align*}
		\sup_{x \in  \VN}\left|h_{\scale_L,\scale_R}(\tfrac{x}{N}) - h^N_{\scale_L,\scale_R}(x) \right|\underset{N\to \infty}\longrightarrow 0\ .
	\end{align*}
	Combined with the integrability of $G \in \mathscr S$, this concludes the proof.	
\end{proof}

In what follows, for each of the three regimes $\tune < 1$, $\tune = 1$ and $\tune > 1$, we conclude the proof of Theorem \ref{theorem:main_hydrodynamics} by proving \eqref{eq:convergence_dynamic_part} for the processes $\left\{\mathscr Z^N_\cdot: N \in \N \right\} \subset \mathcal D([0,\infty),\mathscr S')$  given in \eqref{eq:decomposition_N2}, for all $N \in \N$ and $t \geq 0$, as
\begin{align}\label{eq:zN}
	\mathscr Z^N_t =&\ \frac{1}{N}\sum_{x \in \VN} \delta_{\frac{x}{N}}\, \left(\tfrac{\eta^N_t(x)}{\alpha}- h^N_{\scale_L,\scale_R}(x)  \right)\alpha\ .	
\end{align}
Moreover,  because $h^N_{\scale_L,\scale_R}$ is harmonic for $A^N$, 
we have
\begin{align*}
	\mathcal L^N h^N_{\scale_L,\scale_R}(x)=A^Nh^N_{\scale_L,\scale_R}(x)=0\ ,\quad x \in \VNh\ ,
\end{align*}
and, hence,
\begin{align}\label{eq:drift2}
	\mathcal L^N \langle G, \mathscr Z^N_t\rangle = \frac{1}{N}\sum_{x\in \VN} G(\tfrac{x}{N})\, A^N\left(D_N(\cdot,\eta^N_t)-h^N_{\scale_L,\scale_R}(\cdot) \right)(x)\, \alpha\ .
\end{align}
We further remark that,  because $\left\{\mathscr X^N_t: t \geq 0 \right\}$ and $\left\{\mathscr Z^N_t: t \geq 0 \right\}$ differ only by a deterministic term, the corresponding martingales arising from Dynkin's decomposition coincide.

In conclusion,  from the following identity
\begin{align}\label{eq:identity_D2}
	\nonumber
	&\left(D_N(x,\eta)-h^N_{\scale_L,\scale_R}(x) \right) \left(D_N(y,\eta)-h^N_{\scale_L,\scale_R}(y) \right)\alpha^2\\
	\nonumber
	=&\ D_N(x,y,\eta)\, \alpha\left( \alpha+\1_{\{x=y\}}\right)- D_N(x,\eta)\, h^N_{\scale_L,\scale_R}(y)\, \alpha^2
	- D_N(y,\eta)\, h^N_{\scale_L,\scale_R}(x)\, \alpha^2\\ +&\ h^N_{\scale_L,\scale_R}(x)\, h^N_{\scale_L,\scale_R}(y)\, \alpha^2+ \1_{\{x=y\}}\, D_N(x,\eta)\, \alpha\ ,\qquad x, y \in \VN\ ,
\end{align}
we obtain,  for all $\tune \geq 0$, $N \in \N$, $G \in \mathscr S$ and $t \geq 0$,
\begin{align}\label{eq:identity_square_field}
	\nonumber
	&\ \left(\langle G, \mathscr Z^N_t\rangle\right)^2\\
	\nonumber
	 =&\ \frac{1}{N^2} \sum_{x \in \VN} \sum_{y \in \VN} G(\tfrac{x}{N})\, G(\tfrac{y}{N}) \left(D_N(x,\eta^N_t)-h^N_{\scale_L,\scale_R}(x) \right) \left(D_N(y,\eta^N_t)-h^N_{\scale_L,\scale_R}(y)  \right) \alpha^2\\
	\nonumber
	=&\ \llangle \left(G\otimes G\right)(\tfrac{\cdot}{N},\tfrac{\cdot}{N}),D_N(\cdot,\cdot,\eta^N_t)\rrangle_{N\times N} - 2\, \llangle G(\tfrac{\cdot}{N}),D_N(\cdot,\eta^N_t)\rrangle_N\, \llangle G(\tfrac{\cdot}{N}), h^N_{\scale_L,\scale_R}\rrangle_N \\[.15cm]
	+&\ \left(\llangle G(\tfrac{\cdot}{N}),h^N_{\scale_L,\scale_R}\rrangle_N \right)^2 + \tfrac{1}{N}\llangle G^2(\tfrac{\cdot}{N}),D_N(\cdot,\eta^N_t)\rrangle_N\ .
\end{align}

\subsection{Proof of Theorem \ref{theorem:main_hydrodynamics}}\label{section:proof_HDL}
\subsubsection{Case $\tune \geq 1$}\label{section:proof_HDL_robin_neumann}
Let us recall the definition in \eqref{eq:zN}. In order to prove Theorem \ref{theorem:main_hydrodynamics}, we  show that  the limiting distribution of the  fields $\{\mathscr Z^N_\cdot: N \in \N\}$ is fully supported on solutions of some integral equations;  the uniqueness result in Proposition \ref{proposition:well_posedness}  concludes then the proof.  More specifically, we prove that,
for all $\tune \ge 1$, $\delta > 0$, $T > 0$ and $G \in \mathscr S$, we have
\begin{align}\label{eq:convergence_sup}
	\Pr^N_{\mu^N}\left(\sup_{t \in [0,T]}\left|\langle G, \mathscr Z^N_t\rangle - \int_0^t  \langle \mathcal AG, \mathscr Z^N_s\rangle\, \dd s-\langle G, \left(\scale_0-h_{\scale_L,\scale_R}\dd u\right) \alpha\rangle   \right|>\delta \right)\underset{N\to \infty}\longrightarrow 0\ .
\end{align} 
By Dynkin's formula in \eqref{eq:dynkin_formula},  the above is equivalent to
\begin{align*}
	\Pr^N_{\mu^N}\left(\sup_{t \in [0,T]}\left|\begin{array}{c}\left[\langle G, \mathscr Z^N_0\rangle  - \langle G, \left(\scale_0-h_{\scale_L,\scale_R}\dd u\right) \alpha\rangle\right] \\[.2cm]+ \left[\int_0^t \left(\mathcal L^N\langle G,\mathscr Z^N_s\rangle - \langle \mathcal A G, \mathscr Z^N_s\rangle\right) \dd s\right]\\[.2cm]
		+\left[\langle G, \mathscr M^N_t\rangle \right] \end{array}\right|>\delta \right)\underset{N\to \infty}\longrightarrow 0\ .
\end{align*}
Let us prove that each of the three terms in square brackets vanishes uniformly in $[0,T]$ in probability; to this purpose, we follow some of the arguments in   \cite[Proposition 4.1]{baldasso_exclusion_2017}. 
Regarding the first term containing information only about the initial conditions of the fields and the limiting solution, by assumption \ref{it:assumption_associated} in Theorem \ref{theorem:main_hydrodynamics} and Lemma \ref{lemma:convergence_static_part}, for all $\tune \geq 1$, we have
\begin{align}\label{eq:time_zero_proof}
	\mu^N\left(\left\{\eta \in \mathcal X_N:\left|\begin{array}{r}\frac{1}{N}\sum_{x \in \VN} G(\tfrac{x}{N})\left(\tfrac{\eta(x)}{\alpha}-h^N_{\scale_L,\scale_R}(x) \right)\alpha\\[.2cm] - \int_{[0,1]}G(u)\left(\scale_0(u)-h_{\scale_L,\scale_R}(u)\right) \alpha\, \dd u\end{array} \right|  > \delta \right\} \right)\underset{N\to \infty}\longrightarrow 0\  .
\end{align}
Turning to the second term consisting of a time integral, it suffices to show, by Chebyshev's and Cauchy-Schwarz inequalities, that
\begin{align}\label{eq:integral_term_robin_neumann}
	\sup_{t \in [0,T]}\E^N_{\mu^N}\left[\left(\mathcal L^N \langle G, \mathscr Z^N_t\rangle - \langle\mathcal A G, \mathscr Z^N_t\rangle \right)^2 \right]\underset{N\to \infty}\longrightarrow 0\ .
\end{align}
By \eqref{eq:drift2} and \eqref{eq:computations_laplacian}, we have
\begin{align}\label{eq:bound}
	&\E^N_{\mu^N}\left[\left(\mathcal L^N \langle G, \mathscr Z^N_t\rangle - \langle\mathcal A G, \mathscr Z^N_t\rangle \right)^2 \right]\\
	\nonumber
	\leq&\ 2\,\E^N_{\mu^N}\left[\left(\langle \alpha\Delta_N G- \mathcal AG, \mathscr Z^N_t\rangle \right)^2 \right]\\
	\nonumber
	+&\ 4\,\E^N_{\mu^N}\left[\left(\left(\alpha\, \nabla^+_N G(0)-N^{1-\tune}\alpha_L\, G(\tfrac{1}{N}) \right)\left(D_N(1,\eta^N_t)-h^N_{\scale_L,\scale_R}(1) \right)\alpha\right)^2\right]\\
	+&\ 4\, \E^N_{\mu^N}\left[\left(\left(\alpha\, \nabla^-_N G(1) - N^{1-\tune}\alpha_R\, G(\tfrac{N-1}{N}) \right) \left(D_N(N-1,\eta^N_t)-h^N_{\scale_L,\scale_R}(N-1) \right)\alpha
	\right)^2 \right]\ .
	\nonumber	
\end{align}
The first term on the r.h.s.\ above, by the smoothness of $G \in \mathscr S$, the identity \eqref{eq:identity_square_field} and the upper bound \eqref{eq:bound_second_order_polynomials},
vanishes as $N\to \infty$. Both the second and third terms on the r.h.s.\ in \eqref{eq:bound} can be treated analogously; we therefore only give details for the second term. By \eqref{eq:bound_second_order_polynomials} and the identity \eqref{eq:identity_D2} there exists a constant $C=C(\kappa,\scale_L,\scale_R) > 0$ for which we have
\begin{align*}
	&\E^N_{\mu^N}\left[\left(
	\left(\alpha\,  \nabla_N^+G(0)-N^{1-\tune}\alpha_L\, G(\tfrac{1}{N}) \right)\left(D_N(1,\eta^N_t)-h^N_{\scale_L,\scale_R}(1) \right )\alpha\right)^2\right]\\
	\leq&\ C\, \alpha^2 \left(\alpha\, \nabla_N^+G(0)-N^{1-\tune}\alpha_L\, G(\tfrac{1}{N}) \right)^2\ .
\end{align*}
For the case $\tune = 1$, because $G \in \mathscr S$ is smooth and satisfies the boundary conditions in \eqref{eq:boundary_conditions_robin}, we have
\begin{align*}
	\left(\alpha\, \nabla_N^+G(0)-\alpha_L\, G(\tfrac{1}{N}) \right)^2 \leq&\ 
	2\left(\alpha\, \partial^+_u G(0)-\alpha_L\, G(0) \right)^2
	+ \frac{2}{N^2} \left(\alpha \left\|\partial_u^2 G \right\|_{\infty}+\alpha_L \left\| \partial^2_u G\right\|_\infty \right)^2\\
	=&\ \frac{2}{N^2} \left(\alpha \left\|\partial_u^2 G \right\|_{\infty}+\alpha_L \left\| \partial^2_u G\right\|_\infty \right)^2\underset{N\to \infty}\longrightarrow 0\ ,
\end{align*}
with $\left\| \cdot\right\|_\infty$ denoting the supremum norm on $[0,1]$.
For the case $\tune > 1$, $G \in \mathscr S$ satisfies the boundary conditions in \eqref{eq:boundary_conditions_neumann}, yielding
\begin{align*}
	\left(\alpha\, \nabla_N^+G(0)-N^{1-\tune}\alpha_L\, G(\tfrac{1}{N}) \right)^2 \leq&\ 2\left(\alpha\,\partial^+_u G(0)  + N^{1-\tune}\alpha_L \left\|G \right\|_\infty\right)^2 +\frac{2}{N^2} \left(\alpha \left\|\partial_u G \right\|_\infty \right)^2\\
	=&\ \frac{2}{N^{2(\tune-1)}}\left(\alpha_L\left\| G\right\|_\infty \right)^2 + \frac{2}{N^2} \left(\alpha \left\|\partial_u G \right\|_\infty \right)^2\underset{N\to \infty}\longrightarrow 0\ .
\end{align*}
This proves \eqref{eq:integral_term_robin_neumann} for all $\tune \geq 1$. We conclude the proof of \eqref{eq:convergence_sup} by showing that
\begin{align}\label{eq:martingale_robin_neumann}
	\E^N_{\mu^N}\left[\frac{1}{N^2} \int_0^T	 \left\{\begin{array}{c} \sum_{x \in \VN\setminus \{N-1\}} N^2 \alpha \left(G(\tfrac{x+1}{N})-G(\tfrac{x}{N}) \right)^2 \mathcal V^N_{\{x,x+1\}}(\eta^N_s)\, \alpha\\[.2cm]
		+\, N^{2-\tune} \alpha_L \left(G(\tfrac{1}{N}) \right)^2 \mathcal V^N_{\{0,1\}}(\eta^N_s)\, \alpha\\[.2cm]
		+\, N^{2-\tune} \alpha_R \left(G(\tfrac{N-1}{N}) \right)^2 \mathcal V^N_{\{N-1,N\}}(\eta^N_s)\, \alpha
	\end{array} \right\}\dd s \right]\underset{N \to \infty}\longrightarrow 0\ ,	
\end{align}
where the expression inside the expectation is the predictable quadratic variation of the martingale arising from Dynkin's decomposition of the fields $\{\mathscr Z^N_\cdot: N \in \N\}$, see \eqref{eq:predictable_QV} for the definition.
Indeed, \eqref{eq:martingale_robin_neumann} follows because, by \eqref{eq:bound_second_order_polynomials}, we have 	
\begin{align*}
	\sup_{N\in \N}\sup_{x \in \VNh\setminus\{N\}} \sup_{t \in [0,T]} \left|\E^N_{\mu^N}\mathcal V_{\{x,x+1\}}^N(\eta^N_s) \right|\leq C
\end{align*}
for some constant $C=C(\kappa,\scale_L,\scale_R)> 0$ and, by Fubini, 
\begin{align*}
	&\E^N_{\mu^N}\left[\frac{1}{N^2} \int_0^T	 \left\{\begin{array}{c} \sum_{\substack{x \in \VN\\x \neq N-1}} N^2 \alpha \left(G(\tfrac{x+1}{N})-G(\tfrac{x}{N}) \right)^2 \mathcal V^N_{\{x,x+1\}}(\eta^N_s)\, \alpha\\[.2cm]
		+\, N^{2-\tune} \alpha_L \left(G(\tfrac{1}{N}) \right)^2 \mathcal V^N_{\{0,1\}}(\eta^N_s)\, \alpha\\[.2cm]
		+\, N^{2-\tune} \alpha_R \left(G(\tfrac{N-1}{N}) \right)^2 \mathcal V^N_{\{N-1,N\}}(\eta^N_s)\, \alpha
	\end{array} \right\}\dd s \right]\\
	\leq&\ \frac{C\, T}{N}\left\{\frac{1}{N}\sum_{\substack{x \in \VN\\ x\neq N-1}}  \left(\nabla_N^+ G(\tfrac{x}{N}) \right)^2 \alpha^2 + N^{1-\tune} \alpha_L \left( G(\tfrac{1}{N})\right)^2 \alpha + N^{1-\tune} \alpha_R \left(G(\tfrac{N-1}{N}) \right)^2 \alpha  \right\}\ , 
\end{align*}
which vanishes as $N\to \infty$ because $\tune \geq 1$ and $G \in \mathscr S$ is smooth.

\subsubsection{Case $\tune \in [0,1)$}\label{section:proof_HDL_dirichlet} Here, compared to the case $\tune \geq 1$,  we adopt a different strategy  since, due to the higher intensity of the reservoir interaction, we cannot directly prove the claim in \eqref{eq:convergence_sup} with the supremum over time. Instead, we  prove first convergence of finite dimensional distributions and then tightness for the empirical density fields $\left\{\mathscr Z^N_\cdot: N \in \N \right\}\subset \mathcal D([0,\infty),\mathscr S')$.

\

Let $\left\{\scale(t)-h_{\scale_L,\scale_R}\dd u: t \geq 0 \right\}$ be the unique Dirichlet  solution  with initial condition given by $\scale_0-h_{\scale_L,\scale_R}\dd u$. To the purpose of showing convergence of finite dimensional distributions to those of the deterministic process $\left\{\left(\scale(t)-h_{\scale_L,\scale_R}\dd u\right)\alpha: t \geq 0 \right\} \in \mathcal C([0,\infty),\mathscr S')$, it suffices to prove that, for all $\tune \geq 0$, $t \geq 0$, $G \in \mathscr S$ and $\delta > 0$, 
\begin{align}\label{eq:convergence_one_time_hydrodynamic}
	\Pr^N_{\mu^N}\left(\left|\langle G, \mathscr Z^N_t\rangle- \langle G, \left(\scale(t)-h_{\scale_L,\scale_R}\dd u\right) \alpha\rangle \right| > \delta \right)\underset{N\to \infty}\longrightarrow 0
\end{align}
holds true. Notice again that, compared to \eqref{eq:convergence_sup}, the supremum over time does not appear in the displacement above.
 Instead of proving \eqref{eq:convergence_one_time_hydrodynamic} directly, we introduce an auxiliary process -- reminiscent of the so-called \emph{corrected empirical density field} (see, e.g., \cite{jara_quenched_2008}) --  whose finite dimensional distributions approximate those of the empirical fields $\left\{\mathscr Z^N: N \in \N \right\}$ and for which this convergence follows right away. First, we need to prove the following lemma. 
\begin{lemma}\label{lemma:approximationG}
	For all	 $\tune \in [0,1)$ and $G \in \mathscr S$, there exists a sequence of functions
	\begin{align*}
		\left\{G_N: N \in \N \right\}\ ,\quad  G_N: \VNh/N\to \R\ ,\quad G_N(0)=G_N(1)=0\ ,\quad \forall\, N \in \N\ ,
	\end{align*}
	such that
	\begin{align}\label{eq:convergenceG_N}
		\sup_{x \in \VNh}\left|G_N(\tfrac{x}{N})-G(\tfrac{x}{N}) \right|\underset{N\to \infty}\longrightarrow 0\qquad \text{and}\qquad
		\sup_{x \in \VNh}\left|\mathcal A^NG_N(\tfrac{x}{N})-\mathcal AG(\tfrac{x}{N}) \right|\underset{N\to \infty}\longrightarrow 0
	\end{align}
	hold.
\end{lemma}
\begin{proof}
	The function $G_N$ is given as follows:
	\begin{small}
		\begin{multline*}
			G_N(\tfrac{x}{N})
			\coloneqq \1_{\{x > 0\}}\left\{\begin{array}{l}\tfrac{\alpha}{N^{-\tune}\alpha_L} \left(G(\tfrac{1}{N})-\frac{C_N(G)}{N}\right) \\[.2cm]
				+\1_{\{1<x<N\}}  \left(G(\tfrac{x}{N})-G(\tfrac{1}{N})-\frac{x-1}{N}C_N(G) \right)\\[.2cm] +\1_{\{x=N\}}\left(G(\tfrac{N-1}{N})-G(\tfrac{1}{N})-\frac{N-2}{N}C_N(G)+	 \tfrac{\alpha}{N^{-\tune}\alpha_R} \left(-G(\tfrac{N-1}{N})-\frac{C_N(G)}{N} \right)\right)\end{array}\right\}\ ,
	\end{multline*}\end{small}where
	\begin{align}\label{eq:CGN}
		C_N(G) \coloneqq  \frac{\left(\frac{\alpha}{N^{-\tune} \alpha_L}-1\right)G(\tfrac{1}{N})+G(\tfrac{N-1}{N})\left(1-\frac{\alpha}{N^{-\tune}\alpha_R}\right)}{\frac{1}{N} \left(\frac{\alpha}{N^{-\tune} \alpha_L}+(N-2)+\frac{\alpha}{N^{-\tune}\alpha_R} \right)}
	\end{align}
	is chosen such that $G_N(1)=0$. By applying the generator $A^N$ to such function, we obtain
	\begin{align*}	
		A^NG_N(\tfrac{x}{N})\ &=\ N^2 \alpha \left(G(\tfrac{x+1}{N})+G(\tfrac{x-1}{N})-2 G(\tfrac{x}{N}) \right)	= \alpha\Delta_N G(\tfrac{x}{N})
	\end{align*}
	if $x \in \VN$ and $A^NG_N(0)=A^NG_N(1)=0$.
	As a consequence, we get the second convergence in \eqref{eq:convergenceG_N}.
	On the other side, 
	\begin{align*}
		\sup_{x \in \VNh} \left|G_N(\tfrac{x}{N})-G(\tfrac{x}{N}) \right| \leq \left| G(\tfrac{1}{N})\right|\left(1+\tfrac{\alpha}{N^{-\tune}\alpha_L } \right) + \left|G(\tfrac{N-1}{N})\right|\left(1+\tfrac{\alpha}{ N^{-\tune}\alpha_R} \right)
		\ .
	\end{align*}
	Let us observe that $|G(\tfrac{1}{N})|\leq \tfrac{1}{N}|\partial^+_uG(0)|+ \tfrac{C_0}{N^2}$ and $|G(\tfrac{N-1}{N})|\leq \tfrac{1}{N}|\partial^-_uG(1)|+ \tfrac{C_1}{N^2}$ for all $N \in \N$ large enough and some constants $C_0, C_1 > 0$ independent of $N \in \N$. As a consequence,  because $\tune\in [0,1)$, we obtain the first convergence in  
	\eqref{eq:convergenceG_N}. This concludes the proof.
\end{proof}
Let us now prove convergence in probability  of  one-dimensional distributions, and notice that, by a union bound,  the latter immediately yields convergence  of finite-dimensional distributions.  More precisely, we prove  that, for all $\tune \in [0,1)$, $G \in \mathscr S$,  $t \geq 0$ and $\delta > 0$, we have
\begin{align}\label{eq:convergence_fdd_GN}
	\Pr^N_{\mu^N}\left(\left|\langle G,\mathscr Z^N_t\rangle - \int_0^t \langle \mathcal A G, \mathscr Z^N_s\rangle \, \dd s-\langle G, \scale_0-h_{\scale_L,\scale_R}\dd u\rangle \right|>\delta \right) \underset{N\to \infty}\longrightarrow 0\ .
\end{align} 
By the triangle inequality, the above follows if we can show that, for all $\delta > 0$,
\begin{equation}\label{eq:claim1}
	\Pr^N_{\mu^N}\left(\left|\langle G,\mathscr Z^N_t\rangle - \langle G_N,\mathscr Z^N_t\rangle  \right|>\delta \right) \underset{N\to \infty}\longrightarrow 0
\end{equation}
and
\begin{align}\label{eq:claim2}
	\Pr^N_{\mu^N}\left(\left|\langle G_N,\mathscr Z^N_t\rangle - \int_0^t \langle \mathcal A G, \mathscr Z^N_s\rangle \, \dd s-\langle G, \scale_0-h_{\scale_L,\scale_R}\dd u\rangle \right|>\delta \right) \underset{N\to \infty}\longrightarrow 0
\end{align}
hold, where the functions $\left\{G_N: N \in \N \right\}$ are those given in Lemma \ref{lemma:approximationG}. The claim in \eqref{eq:claim1} follows at once from H\"older's inequality, the uniform bounds in \eqref{eq:bound_second_order_polynomials} and Lemma \ref{lemma:approximationG}.

 Let us now deal with the claim in \eqref{eq:claim2} by means of Dynkin's formula for $\{\langle G_N,\mathscr Z^N_\cdot\rangle: N \in \N\}$. As a first step, we have
\begin{align*}
	&\Pr^N_{\mu^N}\left(\left|\langle G_N,\mathscr Z^N_0\rangle - \langle G, \left(\scale_0-h_{\scale_L,\scale_R}\dd u \right)\alpha \right|>\delta \right)\\
	\leq\ &\Pr^N_{\mu^N}\left(\left|\langle G_N,\mathscr Z^N_0\rangle  - \langle G, \mathscr Z^N_0\rangle \right|>\tfrac{\delta}{2} \right)\\
	+\ &\Pr^N_{\mu^N}\left(\left|\langle G,\mathscr Z^N_0\rangle  - \langle G, \left(\scale_0-h_{\scale_L,\scale_R}\dd u \right)\alpha \right|>\tfrac{\delta}{2} \right)
\end{align*}
and both terms on the r.h.s.\ vanish as $N\to \infty$; more specifically, the first term vanishes because of Markov's inequality, assumption \ref{it:assumption_bounds} and the first convergence in \eqref{eq:convergenceG_N}, while the second term because of assumption \ref{it:assumption_associated} and Lemma \ref{lemma:convergence_static_part}. Moreover, for all $\delta > 0$ and $t \geq 0$, we have
\begin{align*}
	\Pr^N_{\mu^N}\left(\left|\int_0^t \left( \mathcal L^N\langle G_N,\mathscr Z^N_s\rangle - \langle \mathcal A G, \mathscr Z^N_s\rangle \right) \dd s \right|>\delta \right)\underset{N\to \infty}\longrightarrow 0\ ,
\end{align*}
which follows by Markov's inequality, duality \eqref{eq:drift}, the symmetry of $A^N$ as in \eqref{eq:symmetry_one}, Tonelli's theorem, \eqref{eq:bound_second_order_polynomials} and the second convergence in \eqref{eq:convergenceG_N}:
\begin{align*}
	&\Pr^N_{\mu^N}\left(\left|\int_0^t \left( \mathcal L^N\langle G_N,\mathscr Z^N_s\rangle - \langle \mathcal A G, \mathscr Z^N_s\rangle \right) \dd s \right|>\delta \right)\\
	\leq\ &\tfrac{1}{\delta}\,\E^N_{\mu^N}\left[\int_0^t \frac{1}{N}\sum_{x \in \VN} \left|\mathcal A^N G_N(\tfrac{x}{N}) - \mathcal AG(\tfrac{x}{N}) \right| \left(D_N(x,\eta^N_s)+h^N_{\scale_L,\scale_R}(x) \right)\alpha\, \dd s \right]\\
	\leq\ &\tfrac{t}{\delta}\, \sup_{x \in \VNh} \left|\mathcal A^N G_N(\tfrac{x}{N}) - \mathcal AG(\tfrac{x}{N}) \right|\left\{ \sup_{s \in [0,t]} \E^N_{\mu^N}\left[D_N(x,\eta^N_s) \right] + \max\left\{\scale_L,\scale_R \right\}\right\}\ .
\end{align*}
In conclusion, the martingales arising from  Dynkin's decomposition of $\left\{\langle G_N,\mathscr Z^N_\cdot\rangle: N \in \N\right\}$ vanish in probability as $N\to \infty$. Indeed, for all $t \geq 0$, 
\begin{align}
	\E^N_{\mu^N}\left[\left(\langle G_N,\mathscr M^N_t\rangle \right)^2\right] \leq \frac{C\, t}{N}\left\{\frac{1}{N} \sum_{x\in \VN} \left(-\mathcal A^N G_N(\tfrac{x}{N}) \right) G_N(\tfrac{x}{N})\right\}\underset{N\to \infty}\longrightarrow 0\ ,
\end{align}
because the expression between curly brackets is uniformly bounded in $N \to \infty$ by \eqref{eq:convergenceG_N} and where, by \eqref{eq:bound_second_order_polynomials},  $C=C(\kappa,\scale_L,\scale_R)>0$ is a constant independent of $N \in \N$ and $t \geq 0$.

\

The proof of Theorem \ref{theorem:main_hydrodynamics} for the case $\tune \in [0,1)$ ends as soon as we show, by Mitoma's tightness criterion \cite{mitoma_tightness_1983},  that, for all $G \in \mathscr S$, the sequence  $\left\{\langle G, \mathscr Z^N_\cdot\rangle: N \in \N \right\}$ is tight in $\mathcal D([0,\infty),\R)$. Most of the steps of this proof may be adapted from those in Section \ref{section:proof_HDL_robin_neumann}, with the only exceptions that, for all $G \in \mathscr S$ and $t \geq 0$, the following boundary terms
\begin{equation}\label{eq:boundary_term_dirichlet1}
	\E^N_{\mu^N}\left[\left(\alpha\, \nabla^+_N G(0) - \alpha_L\, N^{1-\tune}\, G(\tfrac{1}{N}) \right)^2 \left(D_N(1,\eta^N_t)-h^N_{\scale_L,\scale_R}(1) \right)^2 \alpha^2 \right]
\end{equation}
\begin{equation*}
	\E^N_{\mu^N}\left[\left(\alpha\, \nabla^-_N G(1) - \alpha_R\, N^{1-\tune}\, G(\tfrac{N-1}{N}) \right)^2 \left(D_N(N-1,\eta^N_t)-h^N_{\scale_L,\scale_R}(N-1) \right)^2 \alpha^2 \right]
\end{equation*}
and
\begin{small}
	\begin{equation}\label{eq:boundary_term_dirichlet2}
		\frac{1}{N^2}\left\{ N^{2-\tune} \alpha_L \left(G(\tfrac{1}{N}) \right)^2\E^N_{\mu^N}\left[\mathcal V^N_{\{0,1\}}(\eta^N_t) \right] \alpha  +  N^{2-\tune} \alpha_R\, \left(G(\tfrac{N-1}{N}) \right)^2 \E^N_{\mu^N}\left[\mathcal V^N_{\{N-1,N\}}(\eta^N_t) \right] \alpha \right\}
	\end{equation}
\end{small}are uniformly bounded in $N \in \N$ because of the boundary conditions \eqref{eq:boundary_conditions_dirichlet} that $G \in \mathscr S$ satisfies and  the uniform bounds in \eqref{eq:bound_second_order_polynomials}.

\subsubsection{Some considerations for the case $\tune < 0$} The particle system dynamics described by the generator $\mathcal L^N$ in \eqref{eq:generator} as well as the duality relations and the results in Lemmas \ref{lemma:convergence_static_part} and \ref{lemma:approximationG} clearly extend to the setting of \textquotedblleft fast\textquotedblright\ boundary, i.e., $\tune <0$ if constructing $\mathscr S$ for $\tune < 0$ as done for the case $\tune \in [0,1)$. Moreover, from the first part of the proof in Section \ref{section:proof_HDL_dirichlet}, it follows that, for all $\tune < 0$ and $G \in \mathscr S$, the sequence
\begin{align*}
	\left\{\langle G_N, \mathscr Z^N_\cdot\rangle : N \in \N\right\}
\end{align*}
is tight in $\mathcal D([0,\infty),\R)$, where the sequence $\left\{G_N: N \in \N \right\}$ is the one given in Lemma \ref{lemma:approximationG}, and,  for all $T > 0$ and $\delta > 0$, the following convergence 
\begin{align}\label{eq:uniform_convergence_corrected}
	\Pr^N_{\mu^N}\left(\sup_{t \in [0,T]}\left|\langle G_N, \mathscr Z^N_t \rangle - \langle G, \left(\scale(t)-h_{\scale_L,\scale_R}\dd u \right)\alpha \rangle \right|> \delta  \right)\underset{N\to \infty}\longrightarrow 0
\end{align}
holds, 
where $\left\{\scale(t)-h_{\scale_L,\scale_R}\dd u: t \geq 0 \right\} \in \mathcal C([0,\infty),\mathscr S')$ is the unique  Dirichlet solution in $\mathscr S'$ with initial condition given by $\scale_0-h_{\scale_L,\scale_R}\dd u$. Moreover, by Lemma \ref{lemma:approximationG} and the uniform bounds in \eqref{eq:bound_second_order_polynomials},  it follows that,  for all $\tune < 0$,  $G \in \mathscr S$, $t \geq 0$ and $\delta > 0$, 
\begin{align}\label{eq:convergence_fdd}
	\Pr^N_{\mu^N}\left(\left|\langle G, \mathscr Z^N_t\rangle - \langle G_N,\mathscr Z^N_t\rangle \right|>\delta\right)\underset{N\to \infty}\longrightarrow 0 \ , 
\end{align}
yielding, in particular, convergence of the finite dimensional distribution for the fields $\left\{\mathscr Z^N_\cdot: N \in \N \right\}$: for all $\tune < 0$, $G \in \mathscr S$, $t \geq 0$ and $\delta > 0$,
\begin{align}
	\Pr^N_{\mu^N}\left(\left|\langle G, \mathscr Z^N_t\rangle - \langle G, \left( \scale(t)-h_{\scale_L,\scale_R}\dd u \right)\alpha \rangle  \right|> \delta \right)\underset{N\to \infty}\longrightarrow 0\ .	
\end{align} 
However, tightness of the empirical density fields $\left\{\mathscr Z^N_\cdot: N \in \N \right\}$ in $\mathcal D([0,\infty),\mathscr S')$ for the case $\tune <0$ does not follow from the arguments used in the second part of Section \ref{section:proof_HDL_dirichlet} above because the boundary terms in \eqref{eq:boundary_term_dirichlet1} and \eqref{eq:boundary_term_dirichlet2} are not, in general, uniformly bounded in $N \in \N$. 

An alternative approach to derive the hydrodynamic limit for the case $\tune < 0$ would be, in view of \eqref{eq:uniform_convergence_corrected}, to strengthen the convergence in \eqref{eq:convergence_fdd} by requiring, for all $T > 0$ and $\delta > 0$,
\begin{align}\label{eq:convergence_fdd2}
	\Pr^N_{\mu^N}\left(\sup_{t \in [0,T]}\left|\langle G, \mathscr Z^N_t\rangle - \langle G_N,\mathscr Z^N_t\rangle \right|>\delta\right)\underset{N\to \infty}\longrightarrow 0 \ .
\end{align}
Because of Lemma \ref{lemma:approximationG}, \eqref{eq:convergence_fdd2} would follow, by Markov's inequality, from 
\begin{align}\label{eq:total_number_particles_finite}
	\sup_{N \in \N}\E^N_{\mu^N}\left[\sup_{t \in [0,T]}\left(\frac{1}{N}\sum_{x \in \VN} \eta^N_t(x) \right)^p \right]<\infty
\end{align}
for some $p > 0$.
However, while \eqref{eq:total_number_particles_finite} is trivially satisfied by the $\SIP$ for which each site can accommodate at most one particle at the time, this is no more the case for the open $\SIP$ and  the validity of \eqref{eq:total_number_particles_finite} is not guaranteed.
\subsection{Proof of Theorem \ref{theorem:main_hydrostatic}}\label{section:proof_HDS}

We split the proof of Theorem \ref{theorem:main_hydrostatic} in two parts: we first show that assumption \ref{it:assumption_bounds} and then that assumption \ref{it:assumption_associated} of Theorem \ref{theorem:main_hydrodynamics} hold for the sequence $\left\{\mu^N_{\scale_L,\scale_R}: N \in \N \right\}$. Once these assumptions are verified,   Theorem \ref{theorem:main_hydrodynamics} applies, yielding the hydrostatic limit.  We remark that the arguments employed in this section  hold true also for negative values of the parameter $\tune$.

\

Let us introduce, for all $\tune \geq 0$ and $N \in \N$, the following function
\begin{align}\label{eq:definitionkN}
	k^N_{\scale_L,\scale_R}(x,y)\coloneqq E_{\mu^N_{\scale_L,\scale_R}}\left[D_N(x,y,\eta) \right] = \lim_{t\to \infty}\E^N_{\mu^N}\left[D_N(x,y,\eta^N_t) \right]\ ,\quad x, y \in \VNh\ ,	
\end{align}
for any probability measure $\mu^N$ on $\mathcal X_N$, which, by stationarity of $\mu^N_{\scale_L,\scale_R}$ and duality \eqref{eq:duality_two}, solves the following linear boundary value problem:
\begin{align}\label{eq:boundary_value_discrete_two}
	\left\{\begin{array}{rcll}
		B^Nf(x,y) &=& 0 &\quad (x,y) \in \VN \times \VN\\
		f(x,y) &=& h^N_{\scale_L,\scale_R}(x)\, h^N_{\scale_L,\scale_R}(y) &\quad 	 (x,y) \in (\VNh \times \VNh) \setminus \left(\VN \times \VN\right)\ ,
	\end{array}
	\right.
\end{align}
where we recall that $B^N$ is the infinitesimal generator corresponding to two inclusion particles in $\VN$ with absorbing sites $\{0,N\}$ as defined in \eqref{eq:generator_two_particles} and $h^N_{\scale_L,\scale_R}$ is the solution of \eqref{eq:boundary_value_discrete_one} and given in \eqref{eq:harmonic_discrete}. We note that, while for the open symmetric exclusion process the stationary two-point correlations are known (see, e.g., \cite{derrida_exact_1993-1}, \cite[Eq.\ (2.23)]{goncalves_non-equilibrium_2019}), for the open symmetric inclusion process the function $k^N_{\scale_L,\scale_R}$ is not, in general, explicit.

\subsubsection{Assumption \ref{it:assumption_bounds} of Theorem \ref{theorem:main_hydrodynamics} for the stationary measure}\label{section:proof_HDS_b}
By the maximum principle applied to the boundary value problems \eqref{eq:boundary_value_discrete_one} and \eqref{eq:boundary_value_discrete_two}, we obtain
\begin{align*}
	0\leq h^N_{\scale_L,\scale_R}(x)\leq\max\left\{\scale_L,\scale_R \right\}\qquad \text{and}\qquad 0 \leq k^N_{\scale_L,\scale_R}(x,y)\leq \max\left\{\scale_L^2,\scale_R^2 \right\}
\end{align*}
for all $x, y \in \VNh$, yielding, by \eqref{eq:definitionhN} and \eqref{eq:definitionkN}, the bounds in assumption \ref{it:assumption_bounds} of Theorem \ref{theorem:main_hydrodynamics} with $\kappa = \max\left\{\scale_L,\scale_R \right\}$.

\subsubsection{Assumption \ref{it:assumption_associated} of Theorem \ref{theorem:main_hydrodynamics} for the stationary measure}
In this section we prove that, for all $\tune \geq 0$, $G \in \mathscr S$ and $\delta > 0$, we have
\begin{align}\label{eq:a_stationary}
	\mu^N_{\scale_L,\scale_R}\left(\left\{\eta \in \mathcal X_N: \left|\frac{1}{N}\sum_{x \in \VN}G(\tfrac{x}{N})\,	 \eta(x) - \langle G,h_{\scale_L,\scale_R}\dd u\,\alpha\rangle  \right|>\delta \right\} \right)\ \underset{N\to \infty}\longrightarrow\ 0\ ,
\end{align} 
or, equivalently by Lemma \ref{lemma:convergence_static_part},  
\begin{align}\label{eq:a_stationary2}
	\mu^N_{\scale_L,\scale_R}\left(\left\{\eta \in \mathcal X_N: \left|\frac{1}{N}\sum_{x \in \VN}G(\tfrac{x}{N})\left( D_N(x,\eta)-h^N_{\scale_L,\scale_R}(x)\right)\alpha  \right|>\delta \right\} \right)\ \underset{N\to \infty}\longrightarrow\ 0\ ,
\end{align}
where $h_{\scale_L,\scale_R}\dd u \in \mathcal C^\infty([0,1])$ is given in either \eqref{eq:harmonic_dirichlet} if $\tune < 1$, \eqref{eq:harmonic_robin} if $\tune =1$ or \eqref{eq:harmonic_neumann} if $\tune > 1$. In  view of  Chebyshev's inequality,  we prove
\begin{align}
	E_{\mu^N_{\scale_L,\scale_R}}\left[\left(\frac{1}{N}\sum_{x\in \VN}G(\tfrac{x}{N})\left( D_N(x,\eta)-h^N_{\scale_L,\scale_R}(x)\right)\alpha  \right)^2 \right]\underset{N\to \infty}\longrightarrow 0\ ,	
\end{align}
from which \eqref{eq:a_stationary2} follows for all $\delta > 0$. To this purpose, by \eqref{eq:identity_square_field} and stationarity of $\mu^N_{\scale_L,\scale_R}$, we have
\begin{align*}
	&E_{\mu^N_{\scale_L,\scale_R}}\left[\left(\frac{1}{N}\sum_{x\in \VN}G(\tfrac{x}{N})\left( D_N(x,\eta)-h^N_{\scale_L,\scale_R}(x)\right)\alpha  \right)^2 \right]\\
	=&\  \llangle G\otimes G, k^N_{\scale_L,\scale_R}\rrangle_{N\times N} - \left(\llangle G, h^N_{\scale_L,\scale_R}\rrangle_N \right)^2 + \tfrac{1}{N} \llangle G^2, h^N_{\scale_L,\scale_R}\rrangle_N\\
	=&\ \frac{1}{N^2} \sum_{x, y \in \VN} G(\tfrac{x}{N})\, G(\tfrac{y}{N})\, k^N_{\scale_L,\scale_R}(x,y)\, \alpha \left(\alpha+\1_{\{x=y\}} \right)\\
	-&\ \frac{1}{N^2}\sum_{x,y \in \VN} G(\tfrac{x}{N})\, G(\tfrac{y}{N})\, h^N_{\scale_L,\scale_R}(x)\, h^N_{\scale_L,\scale_R}(y)\, \alpha^2+ \frac{1}{N^2}\sum_{x \in \VN}G(\tfrac{x}{N})^2\, h^N_{\scale_L,\scale_R}(x)\, \alpha\\
	=&\ \frac{1}{N^2} \sum_{x, y \in \VN} G(\tfrac{x}{N})\, G(\tfrac{y}{N})\left(k^N_{\scale_L,\scale_R}(x,y) - h^N_{\scale_L,\scale_R}(x)\, h^N_{\scale_L,\scale_R}(y) \right) \alpha \left(\alpha+\1_{\{x=y\}} \right)\\
	+&\ \frac{1}{N}\left\{\frac{1}{N}\sum_{x \in \VN} G(\tfrac{x}{N})^2 \left(1+h^N_{\scale_L,\scale_R}(x) \right) h^N_{\scale_L,\scale_R}(x)\, \alpha \right\}\ .
\end{align*}
By the uniform boundedness of $G \in \mathscr S$ and $\left\{h^N_{\scale_L,\scale_R}: N \in \N \right\}$, the second term on the r.h.s.\ above vanishes as $N\to \infty$. Hence, we are left only with the proof that
\begin{align}\label{eq:vanish_norm2}
	\nonumber
	&\llangle G\otimes G, k^N_{\scale_L,\scale_R}-h^N_{\scale_L,\scale_R}\otimes h^N_{\scale_L,\scale_R}\rrangle_{N\times N}\\
	=&\ \frac{1}{N^2} \sum_{x, y \in \VN} G(\tfrac{x}{N})\, G(\tfrac{y}{N})\left(k^N_{\scale_L,\scale_R}(x,y) - h^N_{\scale_L,\scale_R}(x)\, h^N_{\scale_L,\scale_R}(y) \right) \alpha \left(\alpha+\1_{\{x=y\}} \right) \underset{N\to \infty}\longrightarrow 0\ .
\end{align}
More specifically, we obtain \eqref{eq:vanish_norm2} from the following upper bound: 
for all $\tune \geq 0$ and $G \in \mathscr S$, we have
\begin{align}\label{eq:vanish_norm22}
	\nonumber
	&\sup_{N \in \N} \max\left\{N,N^{\tune-1} \right\}\left| \llangle G\otimes G, k^N_{\scale_L,\scale_R}-h^N_{\scale_L,\scale_R}\otimes h^N_{\scale_L,\scale_R}\rrangle_{N\times N}\right|\\
	\leq &\sup_{N \in \N} \max\left\{N,N^{\tune-1} \right\}\left\|G \otimes G \right\|_\infty \llangle 1 \otimes 1,\left| k^N_{\scale_L,\scale_R}-h^N_{\scale_L,\scale_R}\otimes h^N_{\scale_L,\scale_R}\right|\rrangle_{N\times N}< \infty\ .
\end{align}
We remark that the upper bound in \eqref{eq:vanish_norm22} differs from those in, e.g., \cite[Eq.\ (3.2)]{landim_stationary_2006} and \cite[Proposition 2.1]{goncalves_non-equilibrium_2019}  derived for the open symmetric exclusion process from the explicit expression of the two-point stationary correlation function and corresponding, in our setting, to
\begin{align*}
	&\sup_{N \in \N} \max\left\{N,N^{\tune-1} \right\}\left| \llangle G\otimes G, k^N_{\scale_L,\scale_R}-h^N_{\scale_L,\scale_R}\otimes h^N_{\scale_L,\scale_R}\rrangle_{N\times N}\right|\\
	\leq &\sup_{N \in \N} \max\left\{N,N^{\tune-1} \right\}\llangle \left|G \otimes G\right|,1\otimes 1\rrangle_{N\times N}  \sup_{(x,y)\in \VN\times \VN}\left| k^N_{\scale_L,\scale_R}(x,y)-h^N_{\scale_L,\scale_R}(x) h^N_{\scale_L,\scale_R}(y)\right|\\
	<&\ \infty\ .	
\end{align*}
In our case, although we do not know, as already mentioned above, the  explicit form of 
\begin{align*}
	k^N_{\scale_L,\scale_R}(x,y) -h^N_{\scale_L,\scale_R}(x)\, h^N_{\scale_L,\scale_R}(y)\ ,\quad x, y \in \VN\ ,
\end{align*}
by \cite[Theorem 3.4]{floreani_boundary2020} and \cite[Lemma 3.5]{floreani_boundary2020} (see also \cite[Remark 3.6(b)]{floreani_boundary2020}), we know the sign of these stationary two-point correlation functions as well as the following representation in terms of absorption probabilities of two inclusion particles. Recalling the notation in Section \ref{section:dual_process}, we have
\begin{align}\label{eq:positivity}
	k^N_{\scale_L,\scale_R}(x,y) -h^N_{\scale_L,\scale_R}(x)\, h^N_{\scale_L,\scale_R}(y) > 0
\end{align}
and
\begin{align}\label{eq:explicit_form}
	\nonumber
	&k^N_{\scale_L,\scale_R}(x,y) -h^N_{\scale_L,\scale_R}(x)\, h^N_{\scale_L,\scale_R}(y) \\
	=&\ \int_0^\infty \sum_{\substack{z \in \VN\\
			z \neq N-1}} \left\{\begin{array}{c}N^2 \left(h^N_{\scale_L,\scale_R}(z+1)-h^N_{\scale_L,\scale_R}(z) \right)^2\\[.2cm] \times\, \widehat \Pr^N_{\xi=\delta_x+\delta_y}\left(\xi^N_s(z)=1\ \text{and}\ \xi^N_s(z+1)=1 \right) \end{array}\right\} \dd s
\end{align}
for all $x, y \in \VN$. 
\begin{remark}
	The above expression for the stationary two-point correlations is related to the stationary solution to the non-homogeneous parabolic difference system in Eqs.\ (2.13)--(2.15) in \cite{goncalves_non-equilibrium_2019} (see also \cite{landim_stationary_2006}). 	However, we remark that, while the solution in \cite{goncalves_non-equilibrium_2019} is obtained by means of Duhamel's principle in terms of the Markov semigroup of two independent random walks, the identity \eqref{eq:explicit_form} is obtained by solving  a \emph{linear} system of evolution equations involving second order duality functions and the Markov semigroup of two interacting dual inclusion particles. The representation of the solution in terms of such  Markov semigroup -- symmetric with respect to $\llangle \cdot, \cdot \rrangle_{N\times N}$ for functions vanishing at the boundary --  will turn out useful later on.
\end{remark}
As a consequence of \eqref{eq:positivity}, we get
\begin{align*}
	&\ \left|\frac{1}{N^2} \sum_{x, y \in \VN} G(\tfrac{x}{N})\, G(\tfrac{y}{N})\left(k^N_{\scale_L,\scale_R}(x,y) - h^N_{\scale_L,\scale_R}(x)\, h^N_{\scale_L,\scale_R}(y) \right) \alpha \left(\alpha+\1_{\{x=y\}} \right)\right|\\
	\leq&\ \left\| G\right\|_\infty^2 \frac{1}{N^2} \sum_{x, y \in \VN} \left(k^N_{\scale_L,\scale_R}(x,y) - h^N_{\scale_L,\scale_R}(x)\, h^N_{\scale_L,\scale_R}(y) \right) \alpha \left(\alpha+\1_{\{x=y\}} \right)\ .
\end{align*}
while, by \eqref{eq:harmonic_discrete} and
\begin{align*}
	N^2\left(h^N_{\scale_L,\scale_R}(z+1)-h^N_{\scale_L,\scale_R}(z) \right)^2\leq \frac{C \left( \scale_L-\scale_R\right)^2}{\max\left\{1, N^{2\tune -2} \right\}}\ ,\quad z \in \VN \setminus \{N-1\}\ ,	
\end{align*}
for some constant $C=C(\alpha, \alpha_L,\alpha_R)>0$, we further obtain
\begin{align*}
	&\frac{1}{N^2} \sum_{x, y \in \VN} \left(k^N_{\scale_L,\scale_R}(x,y) - h^N_{\scale_L,\scale_R}(x)\, h^N_{\scale_L,\scale_R}(y) \right) \alpha \left(\alpha+\1_{\{x=y\}} \right)\\
	\leq&\ \frac{C \left( \scale_L-\scale_R\right)^2}{\max\left\{1, N^{2\tune -2} \right\}}  \frac{1}{N^2} \sum_{x, y \in \VN}\int_0^\infty \sum_{\substack{z \in \VN\\
			z \neq N-1}} \widehat \Pr^N_{\xi=\delta_x+\delta_y}\left(\xi^N_s(z)=1\ \text{and}\ \xi^N_s(z+1)=1 \right)  \alpha \left(\alpha+\1_{\{x=y\}} \right)\dd s\\
	=&\ \frac{C \left( \scale_L-\scale_R\right)^2}{\max\left\{1, N^{2\tune -2} \right\}} \frac{1}{N^2} \sum_{x, y \in \VN}\int_0^\infty   \mathsf P^N\left(\begin{array}{l}( X^{N,x}_s, Y^{N,y}_s) \in \VN \times \VN\\[.15cm] \text{and}\ \left| X^{N,x}_s- Y^{N,y}_s \right|=1\end{array}	 \right)  \alpha \left(\alpha+\1_{\{x=y\}}\right) \dd s\ ,
\end{align*} 
where in the last step we went from an unlabeled to a labeled representation of the dual system consisting of two inclusion particles evolving according to the infinitesimal generator $B^N$ given in \eqref{eq:generator_two_particles} and with $\mathsf P^N$, resp.\ $\mathsf E^N$, denoting the corresponding  law, resp.\ expectation: for all \\ $(x, y) \in \VNh \times \VNh$, 
\begin{align}\label{eq:non-hierarchical}
	\left\{\left(X^{N,x}_t,  Y^{N,y}_t \right): t \geq 0 \right\} \subset \VNh \times \VNh\ ,
\end{align}
denotes the Markov process with generator $B^N$ and initial conditions given by
\begin{align}\label{eq:non-hierarchical2}
	\left( X^{N,x}_0,  Y^{N,y}_0 \right)=(x,y)\quad \text{\normalfont a.s.}\ .
\end{align}
If we let $\left\{ S^N_t: t \geq 0 \right\}$ denote the Markov semigroup associated with the generator $B^N$, then
\begin{align*}
	&\ \frac{C \left( \scale_L-\scale_R\right)^2}{\max\left\{1, N^{2\tune -2} \right\}} \frac{1}{N^2} \sum_{x, y \in \VN}\int_0^\infty   \mathsf P^N\left(\begin{array}{l}( X^{N,x}_s, Y^{N,y}_s) \in \VN \times \VN\\[.15cm] \text{and}\ \left| X^{N,x}_s- Y^{N,y}_s \right|=1\end{array}	 \right)  \alpha \left(\alpha+\1_{\{x=y\}}\right) \dd s\\
	=&\ \frac{C \left( \scale_L-\scale_R\right)^2}{\max\left\{1, N^{2\tune -2} \right\}} \frac{1}{N^2} \sum_{x, y \in \VN}\int_0^\infty    S_s^N f_N(x,y)\, \alpha \left(\alpha+\1_{\{x=y\}} \right) \dd s
\end{align*}
where the function $f_N: \VNh \times \VNh\to \R$ is defined as follows:
\begin{align*}
	f_N(x, y)\coloneqq \begin{dcases} 1 &\text{if}\ x, y \in \VN\ \text{and}\ |x-y|=1\\
		0 &\text{otherwise}\ .
	\end{dcases}
\end{align*}
Moreover, by Tonelli's theorem and by the symmetry of $B^N$ -- and, consequently, of the corresponding semigroup --  with respect to the inner product $\llangle \cdot, \cdot \rrangle_{N\times N}$ for functions vanishing on $(\VNh \times \VNh)\setminus (\VN \times \VN)$ (cf.\ \eqref{eq:symmetryBN}), we obtain
\begin{align*}
	&\ \frac{C \left( \scale_L-\scale_R\right)^2}{\max\left\{1, N^{2\tune -2} \right\}} \frac{1}{N^2} \sum_{x, y \in \VN}\int_0^\infty    S_s^N f_N(x,y)\, \alpha \left(\alpha+\1_{\{x=y\}} \right) \dd s\\
	=&\ \frac{C \left( \scale_L-\scale_R\right)^2}{\max\left\{1, N^{2\tune -2} \right\}} \frac{1}{N^2} \sum_{x, y \in \VN}\int_0^\infty   f_N(x,y)  \left(S_s^N g_N(x,y)\right) \alpha \left(\alpha+\1_{\{x=y\}} \right) \dd s\\
	=&\ \frac{C \left( \scale_L-\scale_R\right)^2}{\max\left\{1, N^{2\tune -2} \right\}}  \int_0^\infty \llangle f_N, S^N_s g_N\rrangle_{N\times N}\, \dd s \ ,
\end{align*}
where the function $g_N: \VNh \times \VNh \to \R$ is the indicator function on $\VN \times \VN \subset \VNh\times \VNh$:
\begin{align}\label{eq:functiong_N}
	g_N(x,y)\coloneqq \1_{\{(x,y) \in \VN \times \VN\} }\ .
\end{align}
By  H\"older's inequality, we have
\begin{equation*}\llangle f_N, S^N_s g_N\rrangle_{N\times N}\leq \llangle f_N,1\rrangle_{N\times N} \sup_{x,y \in \VN} S^N_s g_N(x,y)
\end{equation*}
(all functions	 are non-negative) and
\begin{align*}
	\llangle f_N, 1\rrangle_{N\times N}= \frac{1}{N^2} \sum_{x, y \in \VN} \1_{\{|x-y|=1\}}\, \alpha\left(\alpha+\1_{\{x=y\}} \right) \leq \frac{2\, \alpha^2}{N}\ 	.
\end{align*} As a consequence, we further get
\begin{align}
	\left|\frac{C \left( \scale_L-\scale_R\right)^2}{\max\left\{1, N^{2\tune -2} \right\}} \int_0^\infty \llangle f_N, S^N_s g_N\rrangle_{N\times N}\, \dd s \right| 
	\leq&\  \frac{2\, \alpha^2\,C \left( \scale_L-\scale_R\right)^2 }{\max\left\{N, N^{2\tune -1} \right\}}  \sup_{x, y \in \VN} \int_0^\infty 	S^N_s g_N(x,y)\,  \dd s \ .	
\end{align}
The proof to show that assumption \ref{it:assumption_associated} of Theorem \ref{theorem:main_hydrodynamics} holds for the stationary measures ends if we can show that the r.h.s.\ above vanishes as $N\to \infty$. This last result is the content of the following lemma, whose proof is based on two main ingredients: first, by switching to the system of two inclusion particles to a suitable system of two \textquotedblleft hierarchical\textquotedblright\ first and second class inclusion particles, we provide an upper bound for 
\begin{align*}
	\sup_{x, y \in \VN} \int_0^\infty 	S^N_s g_N(x,y)\,  \dd s
\end{align*}
in terms of an expression involving only the absorption probabilities for a \emph{single} non-interacting particle; then, we conclude by employing the asymptotic result in Lemma \ref{lemma:balordo_rw} below on the absorption probability of the random walk with generator $A^N$ defined in \eqref{eq:generator_A}. 
\begin{lemma}\label{lemma:final_HDS}
	For all $\tune \geq 0$, 
	\begin{align}\label{eq:bound_lemma}
		\sup_{N \in  \N} \frac{1}{\max\left\{1,N^{\tune-1} \right\}}\sup_{x, y \in \VN} \int_0^\infty 	S^N_s g_N(x,y)\,  \dd s < \infty\ .
	\end{align}
	As a consequence, for all $\tune \geq 0$ and $G \in \mathscr S$, 
	\begin{align*}
		\limsup_{N \to \infty} \max\left\{N,N^\tune \right\}  \left|\llangle G\otimes G, k^N_{\scale_L,\scale_R}-h^N_{\scale_L,\scale_R}\otimes h^N_{\scale_L,\scale_R}\rrangle_{N\times N}\right| < \infty\ .
	\end{align*}
\end{lemma}
We present the proof of Lemma \ref{lemma:final_HDS} in the Section \ref{section:first_second_class}, in which we also introduce the  notion of first and second class inclusion particles.

\subsubsection{First \& second class inclusion particles and Proof of Lemma \ref{lemma:final_HDS}}\label{section:first_second_class}

As for the $\SEP$ there is a well-known notion of first class and second class particles (see, e.g., \cite[Part III, p.\ 218]{liggett_stochastic_1999}), we show that an analogue notion exists for the $\SIP$ Roughly speaking, first class particles in the exclusion process evolve regardless of the positions of second class particles and, if their decision is to jump on a site occupied by a second class particle, the latter is \textquotedblleft forced\textquotedblright\ to leave its place  and occupy the place left vacant by the first class particle. In particular, the first class particle evolves as a non-interacting random walk, while the second class particle evolves as an interacting random walk.

Inspired by lookdown constructions available for population genetics models (see, e.g., \cite{donnelly_countable_1996}), a similar picture holds for the $\SIP$. Indeed, while the first class inclusion particle evolves as a non-interacting random walk, the dynamics of the second class inclusion particle is determined by the superposition of two distinct effects:  on the first place, it performs non-interacting random walk jumps and, on the second place, it \textquotedblleft looks down\textquotedblright\ to the first class particle and \textquotedblleft joins\textquotedblright\  it at rate two  if the latter sits at a nearest-neighboring site. In Proposition \ref{proposition:lookdown} we show that, up to average over the role  of first and second class particles at time $t=0$, the distribution at any later time $t > 0$ of an unlabeled hierarchical \textquotedblleft lookdown\textquotedblright\  process coincides with that of  an unlabeled non-hierarchical one. 

On one hand, we recall from \eqref{eq:non-hierarchical} that, for all $x, y \in \VNh$, $\left( X^{N,x}_\cdot, Y^{N,y}_\cdot\right)$ denotes the Markov process on $\VNh \times \VNh$ started from $(x,y)$ and with generator $B^N$ defined in \eqref{eq:generator_two_particles}. We refer to such process as the \emph{non-hierarchical} or \emph{symmetric} process and recall that  $\mathsf P^N$ and $\mathsf E^N$ denote their probability law and corresponding expectation, respectively. On the other hand, we define by $\left(\widetilde X^{N,x}_\cdot, \widetilde Y^{N,y}_\cdot \right)$ the so-called \emph{hierarchical} or \emph{lookdown} Markov process on $\VNh\times \VNh$ started from $(x,y)$ and with generator $C^N$ given, for all functions $f: \VNh \times \VNh\to \R$,	 by
\begin{align}\label{eq:generatorCN}
	\nonumber
	C^N f(x,y)\coloneqq&\ A^Nf(\cdot,y)(x) + A^Nf(x,\cdot)(y)\\
	+&\ \1_{\{x, y \in \VN\}}\1_{\{|x-y|=1 \}} 2\left(f(x,x)-f(x,y) \right)\ .	
\end{align}
We let $\widetilde {\mathsf P}^N$ and $\widetilde{\mathsf E}^N$ denote the probability law and corresponding expectation, respectively. We emphasize that the hierarchical dynamics described by the generator $C^N$ in \eqref{eq:generatorCN} dictates that the interaction part of the dynamics (the second line in the r.h.s.\ in \eqref{eq:generatorCN}) affects only the second class particle and compensates this asymmetry  by doubling the rate of the interaction.
\begin{proposition}\label{proposition:lookdown}
	For all $\tune \in [0,\infty)$,  $N \in \N$,  $(x, y) \in \VNh \times \VNh$ and  $t \geq 0$, we have, for all symmetric functions $f: \VNh \times \VNh\to \R$,
	\begin{align*}
		\mathsf E^N\left[ f\left( X^{N,x}_t,  Y^{N,y}_t \right)\right] = \widetilde {\mathsf E}^N\left[f\left(\widetilde X^{N,U}_t, \widetilde Y^{N,V}_t \right)\right]\ ,
	\end{align*}
	where  the random variables  $(U,V)$  take the values $(x,y) \in \VNh\times \VNh$ or $(y,x)\in \VNh \times \VNh$ with equal probability.
\end{proposition}
\begin{proof}
	As mentioned above, this result is a particular case of the more general lookdown construction for the multi-type  Moran  model with mutation (see, e.g., \cite{donnelly_countable_1996}). However, for the convenience of the reader, we report the short proof below.  Indeed, it suffices to show that, for all symmetric functions $f: \VNh \times \VNh\to \R$ and for all $(x,y) \in \VNh\times \VNh$, we have
	\begin{align*}
		B^N f(x,y) = \tfrac{1}{2}\left(C^N f(x,y) + C^N f(y,x) \right)\ ,	
	\end{align*}
where we recall that the operator $C^N$ was defined in \eqref{eq:generatorCN}.	
	This is indeed the case:
	\begin{align*}
		B^N f(x,y) =&\ A^N f(\cdot,y)(x) + A^Nf(x,\cdot)(y) + \1_{\{x, y \in \VN\}}\, \1_{\{|x-y|=1\}}\left(f(x,x)+f(y,y)-2 f(x,y) \right)\\
		=&\ \tfrac{1}{2}\left(A^Nf(\cdot,y)(x) +A^Nf(y,\cdot)(x)+ A^Nf(x,\cdot)(y) + A^Nf(\cdot,x)(y)	\right)\\
		+&\ \1_{\{x, y \in \VN\}}\, \1_{\{|x-y|=1\}} \left(\left(f(x,x)-f(x,y)\right)+\left(f(y,y)-f(y,x) \right) \right)\\
		=&\ \tfrac{1}{2}\left(C^N f(x,y) + C^N f(y,x) \right)\ .
	\end{align*}
	Because $B^N$ maps symmetric functions into symmetric functions, by induction, a similar identity holds for all $\ell \in \N_0$ and $x, y \in \VNh$, 
	\begin{align*}
		(B^N)^\ell f(x,y) = \tfrac{1}{2}\left((C^N)^\ell f(x,y) + (C^N)^\ell f(y,x) \right)\ ,
	\end{align*}
	yielding, for all $t \geq 0$,
	\begin{align*}
		e^{t B^N} f(x,y)= \tfrac{1}{2}\left(e^{t C^N}f(x,y) + e^{t C^N}f(y,x) \right)\ ,\quad (x,y) \in \VNh \times \VNh\ .
	\end{align*}
\end{proof}
\begin{remark}[\textsc{$n$-class lookdown inclusion particle systems}] One may introduce an analogous hierarchical  \textquotedblleft lookdown\textquotedblright\    construction with more than two, say $n > 2$, inclusion particles, in which the $k$-th class particle ($k \leq n$) evolves  not being affected by the particles of class $\ell > k$ and joins at rate $2$ any neighboring particle in the bulk of class $\ell < k$. Along the same lines,  if the class labels are uniformly randomized at the initial time, then, at any later time, the probability law of the unlabeled  hierarchical coincide with that of the unlabeled non-hierarchical  inclusion process started from the same initial configuration.  However, for our purposes, we only need this equivalence for systems with two particles.	
\end{remark}

\begin{proof}[Proof of Lemma {~\ref{lemma:final_HDS}}]
	In view of Proposition \ref{proposition:lookdown} and because the function $g_N: \VNh \times \VNh\to \R$ defined in \eqref{eq:functiong_N} is symmetric, i.e., $g_N(x,y)=g_N(y,x)$ for all $(x,y) \in \VNh \times \VNh$, we have
	\begin{align*}
		\sup_{x, y \in \VN}& \int_0^\infty 	S^N_s g_N(x,y)\,  \dd s\\ =&\ \sup_{x, y \in \VN} \int_0^\infty 	\mathsf E^N\left[ g_N\left( X^{N,x}_s, Y^{N,y}_s \right)\right]  \dd s\\
		=&\ \sup_{x, y \in \VN} \int_0^\infty 	\tfrac{1}{2}\left( \widetilde {\mathsf E}^N\left[ g_N\left(\widetilde X^{N,x}_s,\widetilde Y^{N,y}_s \right)\right] + \widetilde {\mathsf E}^N\left[ g_N\left(\widetilde X^{N,y}_s,\widetilde Y^{N,x}_s \right)\right]  \right) \dd s\ .
	\end{align*}
	Moreover, by conditioning on the non-absorption of the first class inclusion particle, we further obtain, for all $x, y \in \VNh$, 
	\begin{align*}
		\widetilde {\mathsf E}^N\left[ g_N\left(\widetilde X^{N,x}_s,\widetilde Y^{N,y}_s \right)\right]=&\	 \widetilde {\mathsf P}^N\left(\left(\widetilde X^{N,x}_s, \widetilde Y^{N,y}_s \right) \in \VN \times \VN \right)\\
		=&\ \widetilde {\mathsf P}^N\left(\widetilde Y^{N,y}_s \in \VN\big| \widetilde X^{N,x}_s \in \VN \right) \widetilde {\mathsf P}^N\left(\widetilde X^{N,x}_s \in \VN \right)\\
		\leq&\ \widetilde {\mathsf P}^N\left(\widetilde X^{N,x}_s \in \VN \right)\ , 
	\end{align*}
	which yields
	\begin{align}\label{eq:upper_bound_two_one}
		\nonumber
		\sup_{x, y \in \VN} \int_0^\infty 	S^N_s g_N(x,y)\,  \dd s \leq&\ \sup_{x \in \VN} \int_0^\infty \widetilde {\mathsf P}^N\left(\widetilde X^{N,x}_s \in \VN \right) \dd s\\ =&\ \sup_{x \in \VN} \int_0^\infty \mathsf P^N\left(X^{N,x}_s \in \VN \right) \dd s\ ,
	\end{align}
	where we recall that the law of the first class particle $\widetilde X^{N,x}_\cdot$ coincides, by definition, with that of the random walk $X^{N,x}_\cdot$ on $\VNh$ with generator $A^N$. By  Lemma \ref{lemma:balordo_rw} below, the r.h.s.\ in \eqref{eq:upper_bound_two_one} is bounded above by
	\begin{align*}
		C \max\left\{1, N^{\tune-1}\right\}\ , 
	\end{align*}
	for some constant $C > 0$ independent of $N \in \N$, yielding, in conclusion, \eqref{eq:bound_lemma}.
\end{proof}

\appendix
\section{Construction of test function spaces}\label{appendix:function_spaces}
In this section, we construct the function spaces $\mathscr S$ whose elements serve as test functions	 for the $\mathscr S'$-valued empirical density fields. The setting resembles that in, e.g.,  \cite{franco2015equilibrium, Bernardin2020EquilibriumFF}, although we consider a different family of Hilbertian seminorms which turn $\mathscr S$ into a nuclear Fr\'{e}chet space.

We start by recalling some definitions and facts (\cite[Chapter 11]{kipnis_scaling_1999} and \cite[Chapter 1]{kallianpur_xiong_1995}). Let	
$L^2([0,1])$ be endowed with the standard scalar product $\langle \cdot,\cdot\rangle$,  $\mathcal C^\infty([0,1])$  be the 	 linear subspace of elements of $L^2([0,1])$  with a  smooth representative function on $(0,1)$, whose derivatives are uniformly continuous and, thus, may be continuously extended on $[0,1]$ and  $\mathcal C^\infty_\comp([0,1])$ the subspace of $\mathcal C^\infty([0,1])$ of  compactly supported functions on $(0,1)$. Then
\begin{equation*}
	\mathcal C^\infty_\comp([0,1]) \subset \mathcal C^\infty([0,1]) \subset L^2([0,1])\ ,
\end{equation*} 
with $\mathcal C^\infty_\comp([0,1])$ and, thus, $\mathcal C^\infty([0,1])$ being  dense subspaces of $L^2([0,1])$.
The general framework will be the following: for all $\tune \geq 0$, we consider a densely defined, closed and self-adjoint operator $\mathcal L$ with domain $D(\mathcal L)$ and such that $\langle F, \mathcal L F \rangle \geq 0$ for all $F \in D(\mathcal L)$. Such a self-adjoint operator will arise as associated with a  suitable bilinear form $\left(\mathcal E,D(\mathcal E)\right)$. Moreover, $\{\mathcal T_t: t \geq 0 \}$ will denote the 
semigroup on $L^2([0,1])$ associated with $\mathcal A\coloneqq - \mathcal L$. Then we verify the following property (see \cite[Eq.\ (1.3.17)]{kallianpur_xiong_1995}): 
\begin{equation}\label{eq:hilbert_schmidt_assumption}
	\exists\ k_\ast \in \N\quad \text{such that}\quad \left(I+\mathcal L \right)^{-\frac{k_\ast}{2}}\quad \text{is Hilbert-Schmidt}\ .
\end{equation}
By following the construction in \cite[Example 1.3.2]{kallianpur_xiong_1995}, we get that there exist  $\{\lambda_n : n \in \N_0\} \subset [0,\infty)$ with $0 \leq \lambda_0 \leq \lambda_1 \leq \ldots$  and an orthonormal basis $\{\psi_n: n \in \N_0 \}$ in                                                                                                                                                                                                                    $L^2([0,1])$  such that
\begin{align}\label{eq:eigenfunctions}
	\mathcal  L \psi_n = \lambda_n\psi_n\ ,\quad \forall\, n \in \N_0\ .
\end{align}
Moreover, we define the space
\begin{align}\label{eq:nuclear_space}\nonumber
	\mathscr S \coloneqq&\ \bigcap_{k \in \Z} \left\{F \in L^2([0,1]): \left\| \left(I+\mathcal L \right)^{\frac{k}{2}} F\right\|^2_{L^2([0,1])} < \infty\right\}\\
	=&\ \bigcap_{k \in \Z} \left\{F \in L^2([0,1]): \sum_{n \in \N_0} \left(1+\lambda_n \right)^k \left(\langle F,\psi_n\rangle\right)^2 < \infty\right\}
\end{align}
the inner products on $\mathscr S$ given, for all $k \in \Z$, by
\begin{align*}
	\langle F, G\rangle_k \coloneqq \sum_{n \in \N_0} \left(1+\lambda_n \right)^k \langle F, \psi_n\rangle\, \langle G, \psi_n\rangle\ ,	
\end{align*}
and, for all $k \in \Z	$,  $\mathcal H_k$ as the completion of $\mathscr S$ with respect to $\langle \cdot,\cdot \rangle_k$. Note that, by the assumed density of $D(\mathcal L)$, $\mathcal H_0=L^2([0,1])$. Moreover, for all $k \in \Z$, by Friedrichs extension, we have
\begin{align*}
	\mathcal H_k = D(\left(I+\mathcal L\right)^{\frac{k}{2}})\ .
\end{align*}	 As a consequence of these definitions, 
\begin{equation*}
	\langle F, F\rangle_k \geq	 \langle F, F\rangle_\ell \ , \quad \text{for all}\quad k \geq \ell\ ,
\end{equation*}
and, by \eqref{eq:hilbert_schmidt_assumption}, all the canonical embeddings
$
\mathcal H_k \hookrightarrow \mathcal H_\ell
$
with $k \geq \ell+k_\ast$ are Hilbert-Schmidt. 	This will ensure that $\mathscr S$ endowed with the locally convex topology induced by the family of increasing Hilbertian norms 
\begin{equation}\label{eq:hilbertian_norms}
	\left\{\|\cdot\|_k \coloneqq \sqrt{\langle \cdot, \cdot \rangle_k}: k \in \Z \right\}
\end{equation} is a nuclear Fr\'{e}chet space with topological dual space $\mathscr S'$ given by
\begin{align*}
	\mathscr S' = \bigcup_{k \in \Z} \mathcal H_k
	\ . \end{align*}
Moreover, the semigroup $\left\{\mathcal T_t: t \geq 0 \right\}$ on $L^2([0,1])$ determined by $\mathcal A$ is a strongly continuous contraction semigroup described by
\begin{align*}
	\mathcal T_t F = \sum_{n \in \N_0} e^{-\lambda_n t} \langle F, \psi_n \rangle \psi_n\ ,\quad F \in L^2([0,1])\ ,
\end{align*}
and is \textquotedblleft compatible with $(\mathscr S,L^2([0,1]),\mathscr S' )$\textquotedblright\ \cite[Definition 1.3.5]{kallianpur_xiong_1995} in the following sense: 
\begin{enumerate}[label={\normalfont (\Roman*)},ref={\normalfont (\Roman*)}]
	\item \label{it:p1}  	For all $t \geq 0$, $\mathcal T_t \mathscr S \subseteq \mathscr S$.
	\item \label{it:p2} The restriction $\mathcal T_t\big|_{\mathscr S}: \mathscr S \to \mathscr S$ is continuous for all $t \geq 0$.
	\item \label{it:p3} For all $F \in \mathscr S$, $t \mapsto \mathcal T_t F$ is continuous.
	\item \label{it:p4} $\mathcal A\big|_{\mathscr S}: \mathscr S \to \mathscr S$ is continuous.
\end{enumerate}

\

Given the above common framework, we list below the specific choices of self-adjoint operators $\mathcal L$ and  associated forms  $\left(\mathcal E, D(\mathcal E)\right)$ for each of the three regimes of the parameter $\tune \geq 0$. In what follows, we let   $\mathcal W^{k,p}$ with $k \in \N$, $p \geq 1$     denote the standard Sobolev spaces on $(0,1)$ (see, e.g., \cite{adams2003sobolev}).

\

\noindent \emph{Dirichlet} ($\beta < 1$). For Dirichlet boundary conditions, we consider $\mathcal L$ as the unique self-adjoint operator associated with 
\begin{equation*}
	D(\mathcal E) = \mathcal W^{1,2}_0 \coloneqq \closure{\mathcal C_\comp([0,1])}^{\mathcal W^{1,2}}
\end{equation*}
and
\begin{equation*}
	\mathcal E(F,G) \coloneqq \alpha \int_{[0,1]} 	\tfrac{\dd}{\dd u}  F(u)\, \tfrac{\dd}{\dd u}  G(u)\, \dd u\ .
\end{equation*} 
Moreover, $$D(\mathcal L)  = \left\{F \in \mathcal W^{1,2}_0: \tfrac{\dd^2}{\dd u^2}F \in L^2([0,1]) \right\}$$ (see, e.g., \cite[Example 3.1]{arendt_laplacian_nodate} and references therein).

\

\noindent \emph{Robin} ($\tune =1$). For Robin boundary conditions, we consider $\mathcal L$ as the unique self-adjoint operator associated with
\begin{equation*}
	D(\mathcal E) = \mathcal W^{1,2} \coloneqq \left\{F \in L^2([0,1]): \tfrac{\dd}{\dd u}F \in L^2([0,1]) \right\}
\end{equation*}
and
\begin{equation*}
	\mathcal E(F,G) \coloneqq \alpha\int_{[0,1]} \tfrac{\dd}{\dd u}F(u)\, \tfrac{\dd}{\dd u}G(u)\, \dd u +  \alpha_L\, F(0)\, G(0) - \alpha_R\, F(1)\, G(1)\ .
\end{equation*}
Moreover, 
\begin{align*}
	D(\mathcal L) = \left\{F \in L^2([0,1]): \tfrac{\dd^2}{\dd u^2}F \in L^2([0,1])\ ,\  \tfrac{\dd^+}{\dd u}F(0)=\tfrac{\alpha_L}{\alpha} F(0)\ ,\ \tfrac{\dd^-}{\dd u}F(1)=\tfrac{\alpha_R}{\alpha} F(1) \right\}	
\end{align*}
(see, e.g., \cite{arendt_laplacian_nodate}).

\

\noindent \emph{Neumann} ($\tune > 1$). For Neumann boundary conditions, we consider $\mathcal L$ as the unique self-adjoint operator associated with
\begin{equation*}
	D(\mathcal E) = \mathcal W^{1,2}
\end{equation*}
and
\begin{equation*}
	\mathcal E(F, G) \coloneqq \alpha \int_{[0,1]} \tfrac{\dd}{\dd u}F(u)\, \tfrac{\dd}{\dd u}G(u)\, \dd u\ .
\end{equation*}
Moreover, 
\begin{align*}
	D(\mathcal L)=\left\{F \in \mathcal W^{1,2}: \tfrac{\dd^2}{\dd u^2}F \in L^2([0,1])\ \text{and}\  \eqref{eq:neumann_condition}\ \text{below holds} \right\}\ ,
\end{align*}
where
\begin{align}\label{eq:neumann_condition}
	\int_{[0,1]} \tfrac{\dd^2}{\dd u^2}F(u)\, G(u)\, \dd u = - \int_{[0,1]} \tfrac{\dd}{\dd u}F(u)\, \tfrac{\dd}{\dd u}G(u)\, du\ ,\quad \text{for all}\ G \in \mathcal W^{1,2}\ .
\end{align}
(See, e.g., \cite[Example 3.2]{arendt_laplacian_nodate}).

\

From classical results on the eigenvalues of the Dirichlet, Robin and Neumann Laplacian operators on the interval $[0,1]$ (see, e.g., \cite{NetrusovSafarov05}), we know that, for all $\tune \geq 0$, the self-adjoint operator $\mathcal L$ has a discrete non-negative spectrum. Moreover, by the ordering of Neumann, Robin and Dirichlet eigenvalues (see, e.g., \cite{arendt2003dirichlet}) and by Weyl's law (see, e.g., \cite{NetrusovSafarov05}), if we let, for all $\tune \geq 0$,
\begin{align*}
	\left\{\lambda_n: n \in \N_0 \right\}\ \subset\ [0,\infty)
\end{align*}
denote the eigenvalues associated to the self-adjoint operator $\mathcal L$, there exists a  constant $\varLambda = \varLambda_\tune \in (0,\infty)$ for which  we have:
\begin{align}\label{remark:weyl}
	\frac{\sqrt{\lambda_n}}{n}\ \underset{n \to \infty}\longrightarrow\ \varLambda\ .
\end{align}

As a consequence of \ref{remark:weyl}, we get property \eqref{eq:hilbert_schmidt_assumption} with $k_\ast = 1$.	This property enables the construction of the nuclear Fr\'{e}chet spaces $\mathscr S$ and their topological duals $\mathscr S'$ as above.

\

Let us further characterize such spaces by proving Proposition \ref{proposition:characterization_function_spaces}. 
\begin{proof}[Proof of Proposition ~\ref{proposition:characterization_function_spaces}]
	
	Let us first prove that $\mathscr S$ consists of smooth functions with uniformly continuous derivatives of any order. By \eqref{eq:nuclear_space}, we have	  
	\begin{align*}
		\mathscr S = \bigcap_{k \in \Z} D(\left(I+\mathcal L \right)^{\frac{k}{2}}) \subseteq \mathcal C^\infty([0,1])\ .
	\end{align*}
	Indeed,  the last inclusion is a consequence of 
	\begin{align*}
		D(\left(I+\mathcal L\right)^{\frac{k}{2}})
		\subseteq \mathcal W^{k,2}\ ,\quad k \in \N_0\ ,
	\end{align*}
	and the Sobolev embedding theorems (see, e.g., \cite[Theorem 5.4.II.C']{adams2003sobolev}): for all $k \in \N_0$ and $\lambda \in (0,\frac{1}{2}]$, 
	\begin{align}\label{eq:sobolev_embedding}
		\mathcal W^{k+1,2} \subseteq \mathcal C^{k,\lambda}([0,1])\ ,
	\end{align}
	where $\mathcal C^{k,\lambda}([0,1])$ denotes the subspace of $\mathcal C^k([0,1])$ whose derivatives up to order $k$ are H\"{o}lder continuous with H\"{o}lder exponent $\lambda$ (see, e.g., \cite[\S1.27]{adams2003sobolev}). Moreover, the embedding \eqref{eq:sobolev_embedding} into the Banach space $(\mathcal C^{k,\lambda}([0,1]), \|\cdot\|_{k,\lambda})$  is continuous.
	
	Next,	 let us show which boundary conditions  the test functions satisfy.
	We observe that, for all $\tune \geq 0$, if $G \in \mathscr S$, then
	\begin{equation}\label{eq:G_expansion}
		G = \sum_{n\in\N_0} \langle G, \psi_n\rangle\, \psi_n\ ,
	\end{equation}
	where $\left\{\psi_n : n\in \N_0 \right\}$ denotes the orthonormal basis in $L^2([0,1])$ of eigenfunctions of $\mathcal L$. Moreover, by \eqref{eq:eigenfunctions}, 
	\begin{equation*}
		\left\{\psi_n: n \in \N_0 \right\} \subseteq \bigcap_{k \in \N_0} D(\mathcal L^k) \subseteq \mathscr S\ .
	\end{equation*}
	In particular, because of the definitions of $\mathcal L$ and their domains $D(\mathcal L)$, the eigenfunctions $\left\{\psi_n: n \in \N_0 \right\}$ satisfy the corresponding boundary conditions \eqref{eq:boundary_conditions_dirichlet}--\eqref{eq:boundary_conditions_neumann}. Therefore, if we show that, for all $\tune \geq 0$, $G \in \mathscr S$ and $\ell \in \N_0$, 
	\begin{align*}
		\sum_{n \in \N_0} \left|\langle G,\psi_n\rangle\right| \sup_{u \in [0,1]}\left|\left(\frac{\dd}{\dd u} \right)^\ell \psi_n(u) \right| < \infty\ ,
	\end{align*}
	then, by \eqref{eq:G_expansion}, we get
	\begin{align}\label{eq:series}
		\left(\frac{\dd}{\dd u} \right)^\ell G = \sum_{n \in \N_0} \langle G,\psi_n\rangle\, \left(\frac{\dd}{\dd u} \right)^\ell\psi_n\ ,
	\end{align}
	and the conclusion follows. To this purpose, let us prove that
	\begin{align}\label{eq:inequality_sup_derivative}
		\sup_{u \in [0,1]}\left(\left(\frac{\dd}{\dd u}\right)^\ell\psi_n(u) \right)^2 \leq \sum_{h=0}^\ell \left(\alpha_L + \alpha_R \right)^{2 h} \lambda_n^{\ell+1-h}	
	\end{align}
	holds true for all $\tune \geq 0$, $n \in \N_0$ and $\ell \in \N_0$.
	Indeed,  Cauchy-Schwarz inequality and the boundary conditions satisfied by the eigenfunction $\psi_n$ yield
	\begin{align*}
		\sup_{u \in [0,1]}&\left(\left(\frac{\dd}{\dd u}\right)^\ell\psi_n(u) \right)^2\\
		 &\leq \int_{[0,1]} \left(\left(\frac{\dd}{\dd u}\right)^{\ell+1}\psi_n(u)\right)^2 du\\	
		&=\ \lambda_n^{\ell+1}
		+ \1_{\{\tune =1\}}\left\{-\alpha_L \left(\left(\frac{\dd}{\dd u} \right)^\ell\bigg|_{u=0} \psi_n\right)^2 +\alpha_R \left(\left(\frac{\dd}{\dd u}\right)^\ell\bigg|_{u=1} \psi_n \right)^2 \right\}\\
		&\leq\ \lambda_n^{\ell+1} + \1_{\{\tune =1\}}(\alpha_L+\alpha_R)^2 \left\{\sup_{u \in [0,1]} \left(\left(\frac{\dd}{\dd u} \right)^{\ell-1}  \psi_n(u)\right)^2\right\}\ ,
	\end{align*}
	and, by iterating, we get \eqref{eq:inequality_sup_derivative}.
	As a consequence of \eqref{eq:inequality_sup_derivative} and Cauchy-Schwarz inequality, we get, for all $\{a_n: n \in \N \} \subset (0,\infty)$, 
	\begin{align*} 
		&\sum_{n \in \N_0}  \left|\langle G,\psi_n\rangle\right| \sup_{u \in [0,1]}\left| \left(\frac{\dd}{\dd u} \right)^\ell\psi_n(u)\right|\\
		&\leq\ \sqrt{\sum_{n \in \N_0} \left(\langle G, \psi_n\rangle \right)^2 a_n^2 \sup_{u \in [0,1]} \left(\left(\frac{\dd}{\dd u}\right)^\ell\psi_n(u) \right)^2} \sqrt{\sum_{n \in \N_0} a_n^{-2}}\\
		&\leq\  \sqrt{\sum_{h=0}^\ell \left(\alpha_L+\alpha_R \right)^{2h} \sum_{n \in \N_0} \left(\langle G, \psi_n\rangle\right)^2  a_n^2 \lambda_n^{\ell+1-h}} \sqrt{\sum_{n \in \N_0} a_n^{-2}}\ .
	\end{align*}
	By choosing $\left\{a_n : n \in \N_0 \right\} = \left\{\lambda_n: n \in \N_0 \right\}$, by definition of $G \in \mathscr S$ (cf.\ \eqref{eq:nuclear_space}) and  \ref{remark:weyl}, we obtain the uniform convergence of the series on the r.h.s.\ of \eqref{eq:series}. This concludes the proof.
\end{proof}

\subsection{A remark on the	topologies of $\mathcal C([0,\infty),\mathscr S')$ and $\mathcal D([0,\infty),\mathscr S')$}\label{section:remark_topologies}

We remark that we have defined, for all $\tune \geq 0$, $\mathscr S'$ as the dual of $\mathscr S$ endowed with the strong topology. However,  when considering the spaces $\mathcal C([0,\infty),\mathscr S')$ and $\mathcal D([0,\infty),\mathscr S')$, by  \cite[Lemma 3.2]{perez-abreu_tudor_1992}, the strong topology may be  replaced by the weak$^\ast$ topology (see below) when considering Borel probability measures on such spaces because  the Borel $\sigma$-fields induced by weak$^\ast$ and strong topologies coincide. 
More precisely, let us recall  from \cite{kallianpur_xiong_1995} that:
\begin{enumerate}[label={\normalfont(\arabic*)}]
	\item For all $G \in \mathscr S$, $\|\cdot\|_G : \mathscr S'\to [0,\infty)$ given, for all $f \in \mathscr S'$, by
	\begin{equation*}
		\|f\|_G \coloneqq \left|\langle G, f\rangle \right|\ ,
	\end{equation*}
	is a seminorm. The family $\{\|\cdot\|_G : G \in \mathscr S \}$ determines the weak$^\ast$ topology on $\mathscr S'$ (see \cite[Definition 1.1.3]{kallianpur_xiong_1995}) and, in particular, $\mathscr U \subset \mathscr S'$ is a weak$^\ast$  neighborhood  of $f \in \mathscr S'$ if there exist $n \in \N$, $\{G_1,\ldots, G_n  \} \subset \mathscr S$ and $\{\varepsilon_1,\ldots, \varepsilon_n \} \subset (0,\infty)$ such that
	\begin{align*}
		\mathscr U = \left\{g \in \mathscr S' : \| f-g\|_{G_k} < \varepsilon_k\ \text{for all}\ k = 1,\ldots, n \right\}\ .	
	\end{align*}
	\item For all $G \in \mathscr S$ and $T > 0$, 
	\begin{align*}
		\left\|f \right\|_{G,T} \coloneqq \sup_{t \in [0,T]}\| f(t)\|_G
	\end{align*} 
	with  $f=\{f(t) : t \geq 0 \} \in \mathcal C([0,\infty),\mathscr S')$ 
	defines a seminorm on $\mathcal C([0,\infty),\mathscr S')$. The family $\{\|\cdot\|_{G,T} : G \in \mathscr S, T > 0 \}$ defines the weak$^\ast$ topology of $\mathcal C([0,\infty),\mathscr S')$ (see \cite[p.\ 73]{kallianpur_xiong_1995}), with neighborhoods $\mathscr U \subset \mathscr S'$ of $f=\{f(t): t \geq 0 \} \in \mathcal C([0,\infty),\mathscr S')$  given by finite intersections of sets of the following type:
	\begin{align*}
		\left\{g \in \mathcal C([0,\infty),\mathscr S') : \| f-g\|_{G,T} < \varepsilon \right\}\ .
	\end{align*}
	\item The weak$^\ast$ topology of $\mathcal D([0,\infty),\mathscr S')$ is defined in terms of the following pseudometrics \cite[p.\ 71]{kallianpur_xiong_1995}: for all $G \in \mathscr S$, $T > 0$ and $f, g \in \mathcal D([0,\infty),\mathscr S')$, 
	\begin{align*}
		d_{G,T}(f,g) \coloneqq \inf_{\lambda \in \Lambda_T} \left\{ \sup_{t \in [0,T]}\|f(t)-g(\lambda(t)) \|_G + \gamma(\lambda) \right\}\ ,
	\end{align*}
	where $\Lambda_T$ is the set of strictly increasing continuous maps from $[0,T]$ onto itself and such that
	\begin{align*}
		\gamma(\lambda)\coloneqq \sup_{s,t \in [0,T]} \left|\log\left(\frac{\lambda(t)-\lambda(s)}{t-s}\right) \right| <\infty\ .
	\end{align*}
	Neighborhoods of $f \in \mathcal D([0,\infty),\mathscr S')$ consist of finite intersections of sets of the following type:
	\begin{align*}
		\left\{g \in \mathcal D([0,\infty),\mathscr S') : d_{G,T}(f,g) < \varepsilon \right\}\ .
	\end{align*}
\end{enumerate}

As a consequence of the above definitions and \cite[Lemma 3.2]{perez-abreu_tudor_1992},	 a sequence $\{\mathscr P^N: N \in \N \}$ of Borel probability measures in $\mathcal D([0,\infty),\mathscr S')$ converges in probability to $f \in \mathcal D([0,\infty),\mathscr S')$ if, for all $T > 0$, $G \in \mathscr S$ and $\delta > 0$, 
\begin{align}\label{eq:convergence_probability_1}
	\mathscr P^N\left(\left\{g \in \mathcal  D([0,\infty),\mathscr S') : d_{G,T}(f,g) > \delta \right\} \right) \underset{N\to \infty}\longrightarrow 0\ .
\end{align}
If, in particular, $f \in \mathcal C([0,\infty),\mathscr S')$ and if
\begin{align*}
	\mathscr P^N\left(\left\{g \in \mathcal D([0,\infty),\mathscr S') : \|f-g\|_{G,T} > \delta \right\} \right)\ \underset{N\to \infty}\longrightarrow\ 0
\end{align*}
holds for all $T> 0$, $G \in \mathscr S$ and $\delta > 0$, \eqref{eq:convergence_probability_1} follows.
This notion of convergence in probability to a Dirac measure turns out to be equivalent to  weak convergence in  $\mathcal D([0,\infty),\mathscr S')$. Indeed, this follows from a version of Portmanteau's theorem in the context of completely regular Hausdorff topological spaces  and limiting $\tau$-additive measures (see, e.g., \cite[Corollary II.8.2.4]{bogachev_measure_2007}).
Combined with the above considerations,  a sequence $\{\mathscr P^N: N \in \N \}$ of Borel probability measures on $\mathcal D([0,\infty),\mathscr S')$ converges (either weakly or in probability) to the Dirac measure supported on $f \in \mathcal C([0,\infty),\mathscr S')$ if and only if
\begin{align}\label{eq:convergence_probability}
	\mathscr P^N\left(\left\{g \in \mathcal D([0,\infty),\mathscr S') : \sup_{t \in [0,T]}\left|\langle G,g(t)\rangle-\langle G,f(t)\rangle \right| > \delta \right\} \right)\ \underset{N\to \infty}\longrightarrow\ 0
\end{align}
holds for all $T > 0$, $G \in \mathscr S$ and $\delta > 0$.

\section{Absorbing random walk's estimate}\label{appendix:RW}
In order to study  absorption probabilities  before a given time for the dual random walk $$\left\{X^{N,x}_t: t \geq 0 \right\}$$ on $\VNh$ with  generator $A^N$ given in \eqref{eq:generator_A},  we  employ Stone's pathwise construction of birth-and-death processes from a time-change of Brownian motion paths (see  \cite{stone1963}). To this purpose,  $\mathsf P$ and $\mathsf E$ denote the probability law and corresponding expectation of the underlying one-dimensional standard Brownian motion $\left\{B_t: t \geq 0 \right\}$  with $B_0=0$ a.s.\ and 	 $\mathsf E\left[ (B_t)^2\right]=t$. Let us briefly describe such construction and introduce some notation.

We first define a (singular with respect to Lebesgue) measure $\nu^N$ on $\R$ which yields the correct  time-change of Brownian motion paths. More precisely, $\nu^N$ has the following form
\begin{equation*}
	\nu^N \coloneqq \sum_{x \in \Z} w^N_x \delta_{z^N_x}\ , 
\end{equation*}
where $\{w^N_x: x \in \Z \}$ are called the \textquotedblleft weights\textquotedblright\ and are given by
\begin{align*}
	w^N_x &= \frac{\alpha}{2 N}\ , \quad x \in \Z\ ,
\end{align*}
while $\{z^N_x: x \in \Z \} \subset \R$, satisfying the order relation $z^N_x < z^N_{x+1}$ for all $x \in \Z$ and given, for $x \in\VNh$, by 
\begin{eqnarray}\label{eq:definition_zN}
	\nonumber
	z^N_0&=& 0 \\
	\nonumber
	z^N_1 &=& \frac{1}{ \alpha_L \alpha N^{1-\tune}}\\
	\nonumber
	&\vdots&\\
	\nonumber
	z^N_x &=&\frac{1}{ \alpha_L \alpha N^{1-\tune}}+ (x-1)\, \frac{1}{\alpha^2 N}\\
	\nonumber
	&\vdots& \\
	z^N_N &=& \frac{1}{ \alpha_L \alpha N^{1-\tune}}+(N-2)\, \frac{1}{\alpha^2 N} +  \frac{1}{ \alpha_R \alpha N^{1-\tune}}\ ,
\end{eqnarray} stands for the \textquotedblleft support\textquotedblright\ of $\nu^N$. The specific choice of the support points $\{z^N_x: x \in \Z \setminus \VNh\}$ is irrelevant for our purposes. Let us note  that, by \eqref{eq:definition_zN}, there exists a constant $C >0$ such that, for all $\tune \in \R$, 
\begin{align}\label{eq:bound_range}
	0  < \inf_{N \in \N}  \frac{z^N_N-z^N_0}{\max\left\{1,  N^{\tune-1}\right\}}\leq  \sup_{N \in \N} \frac{z^N_N-z^N_0}{\max\left\{1,  N^{\tune-1}\right\}} \leq \frac{C}{2	\alpha}\ .
\end{align}

Let $\{\ell_t^{N,x}(z): (t, z) \in [0,\infty) \times \R\}$ denote the local time of $\{B_t+z^N_x: t \geq 0\}$  (see, e.g., \cite[Theorem (Trotter)]{stone1963}). Hence, 
\begin{equation*}
	\psi^{N,x}_t = \int_\R \ell_t^{N,x}(z)\, \nu^N(\dd z)
\end{equation*}
is
the random $\nu^N$-weighted time that the Brownian motion has spent on the support of $\nu^N$ up to time $t \geq 0$. We note that $\psi^{N,x}_\cdot: [0,\infty)\to [0,\infty) $ is a non-negative non-decreasing function. 
As a consequence of \cite[\S3]{stone1963}, the process $\{Z^{N,x}_t:	 t \geq 0\}$ defined (a.s.)\ as
\begin{equation*}
	Z^{N,x}_t \coloneqq B_{\phi^{N,x}_t}+z^N_x\ ,
\end{equation*}
with $\{\phi^{N,x}_t: t \geq 0 \}$ being the generalized inverse of $\{\psi^{N,x}_t: t \geq 0 \}$, namely
\begin{equation*}
	\phi^{N,x}_t = \sup\{s \geq 0 : \psi^{N,x}_s \leq t \}\ ,
\end{equation*}
is a jump process on $\{z^N_y,\ y \in \Z \}$ with nearest-neighbor jumps, starting from $z^N_x \in \R$,  exit rates at $z^N_y$ given by
\begin{equation*}
	\frac{1}{2 w^N_y} \frac{z^N_{y+1}-z^N_{y-1}}{(z^N_{y+1}-z^N_y)(z^N_y-z^N_{y-1})}
\end{equation*}
and jump probability from $z^N_y$ to $z^N_{y-1}$ given by
\begin{equation*}
	\frac{z^N_{y+1}-z^N_y}{z^N_{y+1}-z^N_{y-1}}\ .
\end{equation*}
In particular,  the law of the process $\{Z^{N,x}_t: t \geq 0 \}$  coincides with that of the process 
\begin{align*}
	\left\{z^N_{X^{N,x}_t}: t \geq 0 \right\}
\end{align*}
if we observe both processes until the first hitting of $\{z^N_0,z^N_N\}$. Ultimately, this construction stands at the core of the proof of Lemma \ref{lemma:balordo_rw} because it allows us to write random walks' probabilities in terms of Brownian motion probabilities
\begin{equation}\label{eq:identity_absorption_probability}
	\mathsf P^N\left(X^{N,x}_t \in \VN\right) = \mathsf P\left( \psi^{N,x}_{\tau^{N,x}} >	 t \right)\ ,
\end{equation}
where $\tau^{N,x}$ denotes the first exit time from $(z^N_0, z^N_N) \subset \R$ of  $\left\{B_t+ z^N_x: t \geq 0\right\}$. 
\begin{lemma}\label{lemma:balordo_rw}
	There exists a constant $C>0$ such that, for all $\tune \in \R$ and  $N \in \N$, we have
	\begin{align*}
		\sup_{x \in \VN}\int_0^\infty \mathsf P^N\left( X^{N,x}_t\in \VN\right)\dd t \leq  C \max\left\{1, N^{\tune -1} \right\}\  .		
	\end{align*}
\end{lemma}
\begin{proof}
	In view of the identity in \eqref{eq:identity_absorption_probability} and because $\psi^{N,x}_{\tau^{N,x}}$ is a non-negative random variable, we  have, for all $\tune \in \R$,  $N \in \N$ and $x \in \VN$,
	\begin{align}\label{eq:identity_expectation}
		\int_0^\infty \mathsf P^N\left( X^{N,x}_t \in \VN \right)\dd t &= \int_0^\infty \mathsf P\left( \psi^{N,x}_{\tau^{N,x}} >	 t \right) \dd t = \mathsf E\left[\psi^{N,x}_{\tau^{N,x}} \right]\ .
	\end{align}
	By the definitions of $\nu^N$, $\psi^{N,x}$ and $\phi^{N,x}$ above, we have
	\begin{align*}
		\psi^{N,x}_{\tau^{N,x}} = \frac{\alpha}{2N} \sum_{y=1}^{N-1} \ell^{N,x}_{\tau^{N,x}}(z^N_y)\ ,
	\end{align*}
	and, thus,
	\begin{align}\label{eq:expectation_gathered_time}
		\mathsf E\left[\psi^{N,x}_{\tau^{N,x}} \right] = \frac{\alpha}{2N} \sum_{y\in \VN} \mathsf E\left[\ell^{N,x}_{\tau^{N,x}}(z^N_y) \right]\ .	
	\end{align}
	Because the local times are non-negative random variables,  we get, for all $N \in \N$ and $y \in \VN$,
	\begin{align}\label{eq:split}\nonumber
		\mathsf E\left[\ell^{N,x}_{\tau^{N,x}}(z^N_y) \right] =&\ \int_0^\infty
		\mathsf P\left(\ell^{N,x}_{\tau^{N,x}}(z^N_y) > t	\right)\dd t\\
		\nonumber 
		=&\ \int_0^\infty \mathsf P\left(\ell^{N,x}_{\tau^{N,x}}(z^N_y) > t\ \text{and}\ B_{\tau^{N,x}}+z^N_x= z^N_0	\right)\dd t\\ 
		+&\ \int_0^\infty \mathsf P\left(\ell^{N,x}_{\tau^{N,x}}(z^N_y) > t\ \text{and}\ B_{\tau^{N,x}}+z^N_x= z^N_N	\right) \dd t\ .	
	\end{align}
	Let us provide, for all $N \in \N$, an upper bound uniform in $x$ and $y \in \VN$ for the first term on the r.h.s.\ above. To this purpose, we employ \cite[Formula 3.3.6(a), p.\ 214]{borodin_handbook_2002}: for all $x, y \in \VN$,  
	\begin{align}\label{eq:formula_borodin}
		\mathsf P	\left(\ell^{N,x}_{\tau^{N,x}}(z^N_y) > t\ \text{and}\ B_{\tau^{N,x}}+z^N_x= z^N_0	\right)\ =\  C^{N,x}_y \exp\left(- D^N_y t \right)\ ,
	\end{align}
	where
	\begin{align*}
		C^{N,x}_y\ \coloneqq\ \begin{dcases}
			\frac{z^N_N-z^N_x}{z^N_N-z^N_0} &\text{if}\  y < x\\
			\frac{\left(z^N_x-z^N_0 \right)\left(z^N_N-z^N_y \right)}{\left(z^N_y-z^N_0 \right)\left(z^N_N-z^N_0 \right)} &\text{if}\ y \geq x\ ,
		\end{dcases}
	\end{align*}
	and
	\begin{align*}
		D^N_y\ \coloneqq\ \frac{z^N_N-z^N_0}{2\left(z^N_N-z^N_y \right)\left( z^N_y -  z^N_0\right)}\ .
	\end{align*}
	By integrating over time the expression in \eqref{eq:formula_borodin}, we get
	\begin{align*}
		&\int_0^\infty \mathcal 	P\left(\ell^{N,x}_{\tau^{N,x}}(z^N_y) > t\ \text{and}\ B_{\tau^{N,x}}+z^N_x= z^N_0	\right) \dd t\\ 
		&=\ 2 \begin{dcases}
			\frac{z^N_N-z^N_x}{\left(z^N_N-z^N_0 \right)^2} \left(z^N_N-z^N_y \right)\left(z^N_y-z^N_0 \right) &\text{if}\ y < x\\
			\frac{z^N_x-z^N_0}{\left(z^N_N-z^N_0 \right)^2}\left(z^N_N-z^N_y \right) &\text{if}\ y \geq x\ ,
		\end{dcases}
	\end{align*}
	from which we obtain the following upper bound, uniform in $x$ and $y \in \VN$:
	\begin{align}\label{eq:bound1}
		\int_0^\infty \mathcal 	P\left(\ell^{N,x}_{\tau^{N,x}}(z^N_y) > t\ \text{and}\ B_{\tau^{N,x}}+z^N_x= z^N_0	\right) \dd t \leq 2 \left(z^N_N-z^N_0 \right)\ .
	\end{align}
	An analogous argument yields
	\begin{align}\label{eq:bound2}
		\int_0^\infty \mathcal 	P\left(\ell^{N,x}_{\tau^{N,x}}(z^N_y) > t\ \text{and}\ B_{\tau^{N,x}}+z^N_x= z^N_N	\right) \dd t \leq 2 \left(z^N_N-z^N_0 \right)\ ,
	\end{align}
	for all $x$ and $y \in \VN$. By \eqref{eq:identity_expectation}, \eqref{eq:expectation_gathered_time}, \eqref{eq:split}, \eqref{eq:bound1} and \eqref{eq:bound2}, we get:
	\begin{align*}
		\int_0^\infty \mathsf P^N\left(X^{N,x}_t \in \VN \right)\dd 	t \leq 2\alpha \left(z^N_N-z^N_0 \right)\ .
	\end{align*}
	The upper bound in \eqref{eq:bound_range} concludes the proof.	
\end{proof}

\section*{Acknowledgements}
F.S.\ wishes to  thank Joe P.\ Chen for some fruitful discussions at an early stage of this work. C.F. and P.G. thank  FCT/Portugal for support through the project 
UID/MAT/04459/2013.  This project has received funding from the European Research Council (ERC) under  the European Union's Horizon 2020 research and innovative programme (grant agreement   No.\ 715734). F.S. thanks   CAMGSD, IST, Lisbon, where part of this work has been done, and the European research and innovative programme No.\ 715734 for the kind hospitality. F.S.\ was founded by the European Union's Horizon 2020 research and innovation programme under the Marie-Sk\l{}odowska-Curie grant agreement  No.\ 754411.

\end{document}